\documentclass{article}

\usepackage{arxiv}

\usepackage[utf8]{inputenc} 
\usepackage[T1]{fontenc}    
\usepackage{hyperref}       
\usepackage{url}            
\usepackage{booktabs}       
\usepackage{amsfonts}       
\usepackage{nicefrac}       
\usepackage{microtype}      
\usepackage{lipsum}		
\usepackage{graphicx}
\usepackage{natbib}
\usepackage{doi}

\input{src/preamble_arxiv}

\title{Consistency of Nonparametric Density Estimators in CAT(0) Orthant Space}

\author{{\hspace{1mm}Yuki Takazawa} \\
	Graduate School of Information Science and Technology \\ The University of Tokyo \\
        Tokyo, Japan \\
	\texttt{yuki-takazawa@g.ecc.u-tokyo.ac.jp} \\
	\And
	\href{https://orcid.org/0000-0003-2548-0314}{\includegraphics[scale=0.06]{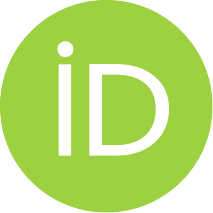}\hspace{1mm}Tomonari Sei} \\
	Graduate School of Information Science and Technology \\ The University of Tokyo \\
    Tokyo, Japan \\
	\texttt{sei@mist.i.u-tokyo.ac.jp} \\
}

\hypersetup{
pdftitle={Consistency of Nonparametric Density Estimators in CAT(0) Orthant Space},
pdfsubject={stat.TH, stat.ME},
pdfauthor={Yuki Takazawa, Tomonari Sei},
pdfkeywords={BHV treespace, kernel density estimation, log-concave density estimation, consistency, nonparametric statistics},
}

\begin{document}
\maketitle

\begin{abstract}
The inference of evolutionary histories is a central problem in evolutionary biology. The analysis of a sample of phylogenetic trees can be conducted in Billera-Holmes-Vogtmann tree space, which is a CAT(0) metric space of phylogenetic trees. The globally non-positively curved (CAT(0)) property of this space enables the extension of various statistical techniques. In the problem of nonparametric density estimation, two primary methods, kernel density estimation and log-concave maximum likelihood estimation, have been proposed, yet their theoretical properties remain largely unexplored.

In this paper, we address this gap by proving the consistency of these estimators in a more general setting—CAT(0) orthant spaces, which include BHV tree space. We extend log-concave approximation techniques to this setting and establish consistency via the continuity of the log-concave projection map. We also modify the kernel density estimator to correct boundary bias and establish uniform consistency using empirical process theory.
\end{abstract}


\section{Introduction}

The inference of evolutionary histories is a fundamental problem in biology. Over the years, numerous methods have been developed to reconstruct phylogenetic trees from gene sequence data. The outcome of such analyses is often not a single phylogenetic tree but a collection of trees. A key example is gene tree estimation, where each tree represents the evolutionary history of an individual gene. Because genes do not necessarily share the same evolutionary history, conflicts among gene trees are common, necessitating methods to reconcile them. Another common scenario is bootstrapping, which generates multiple trees as bootstrap estimates, or Bayesian inference, where posterior draws of trees are typically obtained using Markov chain Monte Carlo (MCMC) algorithms.

Traditionally, consensus approaches have been applied to sets of trees to generate a single representative tree. Another approach is to use Billera-Holmes-Vogtmann (BHV) tree space \cite{Billera2001-vl} as the space of phylogenetic trees and to treat a collection of trees as a sample in this space. BHV tree space consists of many Euclidean nonnegative orthants glued together along their boundaries. This is not a standard Euclidean space, which complicates the application of traditional statistical methodologies. However, this space is known to be a globally non-positively curved metric space (CAT(0) space). This property, in particular, ensures the unique existence of geodesics—the shortest paths between two points in this space. Efficient algorithms to compute them are also available \cite{Owen2011-nq}. Several statistical methods have been extended to this space using the CAT(0) property: estimation of the Fr\'echet mean \cite{Benner2014-xx}, construction of confidence sets \cite{Willis2019-fx}, and principal component analysis \cite{Nye2011-is, Nye2017-vx}, among others. Because the complex geometry of BHV tree space makes parametric modeling difficult, most of the statistical methods developed in this setting are nonparametric.

Density estimation provides a way to characterize the distribution of tree samples in BHV tree space, enabling analyses that go beyond representing them by a single consensus tree. Such approaches open up possibilities for detecting biologically meaningful patterns. For example, density estimation has been used to identify outlier gene trees in phylogenomic studies, flagging trees in low-density regions that often correspond to biological anomalies such as horizontal gene transfer or alignment errors. Identifying such trees can help improve downstream analyses such as species tree inference, by reducing potential biases introduced by anomalous data.

Two nonparametric approaches to density estimation in tree space have been proposed: kernel density estimation using Gaussian-type kernels \cite{Weyenberg2014-wd, Weyenberg2017-kh} and the log-concave maximum likelihood method, recently extended to tree space \cite{Takazawa2024-ed}. In practice, kernel density estimation is feasible for moderately large numbers of taxa, though it requires approximations to the normalizing constants. By contrast, current implementations of the log-concave maximum likelihood estimator are restricted to very low-dimensional cases, such as trees with three or four taxa.

Despite these developments, little is known about the theoretical properties of either estimator. In this paper, we establish conditions under which kernel density estimators and log-concave maximum likelihood estimators are consistent in the broader class of CAT(0) orthant spaces, which includes BHV tree space as a special case. For the log-concave maximum likelihood estimator, we use the notion of log-concave approximation proposed by \cite{Dumbgen2011-mq} in Euclidean space and prove consistency via the continuity of the log-concave projection map. For the kernel density estimator, the original form proposed by \cite{Weyenberg2014-wd} suffers from boundary bias; we introduce a modified version to remove this issue and establish its uniform consistency using tools from empirical process theory \cite{Gine2001-be}.

\section{Preliminaries}
\label{sec:preliminaries}
    \subsection{CAT(0) orthant space}
    \subsubsection{Orthant space}

    First, we define the orthant space considered in this study, following \cite{Miller2015-jm}. Let $\mathcal{E}$ be a finite set representing the set of ``axes,'' and let $\Omega \subseteq 2^{\mathcal{E}}$ be a simplicial complex. We call $\Omega$ a simplicial complex if for any $F \in \Omega$ and any subset $G \subseteq F$, we also have $G \in \Omega$.
     Additionally, we assume that every singleton set $\{e\}$ with $e \in \mathcal{E}$ belongs to $\Omega$.
     Now, to each $F \in \Omega$, we associate the nonnegative orthant $\bar{O}_F \coloneqq \mathbb{R}^F_{\geq 0}$. Here, $\mathbb{R}^F_{\geq 0}$ is the $|F|$-dimensional nonnegative orthant generated by the axes in $F$. A point $x \in \mathbb{R}^{|F|}_{\geq 0}$ in $\bar{O}_F$ (denoted as $(x,F)$) is regarded as having zero coordinates on all axes in $\mathcal{E} \setminus F$. Two points $(x,F)$ and $(y,G)$ are identified if their nonzero coordinates match for all axes in $\mathcal{E}$. Thus, two nonnegative orthants $\bar{O}_F$ and $\bar{O}_G$ are regarded as glued together along the orthant $\bar{O}_{F\cap G}$ when $F \cap G \neq \emptyset$. The orthant space $\mathcal{O}(\mathcal{E},\Omega)$ is the space constructed by gluing such nonnegative orthants along common faces. We also denote by $O_F \coloneqq \mathbb{R}^F_{>0}$ the positive orthant with axes $F$.

\begin{example}\label{ex:k-spider}
    The simplest orthant space is when $\Omega$ consists only of singleton sets and the empty set. For example, if we define $\mathcal{E} = \{1,2,3\}$ with the axes indexed by natural numbers and set $\Omega = \{\emptyset, \{1\}, \{2\}, \{3\}\}$, the corresponding orthant space $\mathcal{O}(\mathcal{E},\Omega)$ is the space shown in Figure \ref{fig:twodim_treespace} (left). Since it consists of three half-lines connected at the origin, this space is called a {\it $3$-spider}. More generally, a space with $k$ connected half-lines is called a {\it $k$-spider}. As discussed later, the $3$-spider space represents the space of rooted phylogenetic trees with three taxa.
\end{example}



For the orthant space $\mathcal{O}(\mathcal{E},\Omega)$, we define its dimension as $\max_{F \in \Omega} |F|$. The point at which all coordinates are zero is called the origin. The $k$-spider from Example \ref{ex:k-spider} is an example of a one-dimensional space. For an orthant space $\mathcal{O}(\mathcal{E}, \Omega)$ of dimension $p$, we also define $\Omega_p = \{F \in \Omega \mid |F| = p\}$ to be the set of all faces of maximum dimension.

An important example of orthant spaces in applications is the space of phylogenetic trees constructed by \cite{Billera2001-vl}.
\begin{example}[BHV Treespace]\label{ex:treespace}
A tree with $n+1$ labeled leaves is called an $n$-tree. An $n$-tree can be viewed as a phylogenetic tree with a single ``root'' representing a common ancestor and $n$ leaves representing extant species. Phylogenetic trees are characterized by their topology and branch lengths. In the construction of phylogenetic tree space by \cite{Billera2001-vl}, only internal branch lengths are considered. Here, a branch is internal if it does not connect directly to a leaf; otherwise it is external. If external branch lengths are also to be modeled, one may work in the product space of BHV tree space and $\mathbb{R}^n$.

An internal branch induces a cluster of leaves, corresponding to the taxon set beneath the subtree attached to that branch. The collection of all such clusters (i.e., the set of nontrivial taxon subsets) corresponds to the set of coordinate axes in the orthant space, denoted as $\mathcal{E}^{(n)}$. In what follows, we identify each internal branch with its induced cluster whenever no ambiguity arises. A set of internal branches (or the induced clusters) $F$ is said to be compatible if they can coexist within the same tree. In that case, we set $F \in \Omega^{(n)}$. It is straightforward to see that $\Omega^{(n)}$ forms a simplicial complex. The resulting orthant space, denoted as $\mathcal{T}_n = \mathcal{O}(\mathcal{E}^{(n)}, \Omega^{(n)})$, defines the phylogenetic tree space. 

Each orthant of $\mathcal{T}_n$ encodes a fixed tree topology, and the coordinates within that orthant record the internal branch lengths of that topology. Two orthants meet along a common face when the corresponding topologies become identical after contracting some branches to zero length. The origin, shared by all orthants, represents the fully unresolved star tree with no internal branches.

    For a binary $n$-tree, there are $n-2$ internal branches, which is the maximum number of coexisting branches in an $n$-tree. Hence, the dimension of this orthant space $\mathcal{T}_n$ is $n-2$. Moreover, since the number of distinct topologies of binary trees is $(2n-3)!!$ \citep{Felsenstein2004-yp}, there are $(2n-3)!!$ orthants of dimension $n-2$ glued together. In other words, $|\Omega^{(n)}_{n-2}| = (2n-3)!!$.

    When $n=3$, the maximum dimension is one, and there are three distinct topologies of binary trees. Thus, the phylogenetic tree space $\mathcal{T}_3$ coincides with the $3$-spider space from Example \ref{ex:k-spider}. When $n=4$, the space is two-dimensional, with 15 distinct topologies of binary trees. Hence, $\mathcal{T}_4 = \mathcal{O}(\mathcal{E}^{(4)}, \Omega^{(4)})$ consists of 15 two-dimensional orthants glued together. The structure of $\Omega^{(4)}$ is shown in Figure \ref{fig:twodim_treespace} (center). Although the space $\mathcal{T}_4$ is too complex to depict fully, the structure of $\Omega^{(4)}$ shows that locally it consists of three nonnegative orthants joined along a common axis (Figure \ref{fig:twodim_treespace}, right). 
\end{example}

\begin{figure}[t]
\centering
\begin{minipage}{0.3\linewidth}
\centering
        \includegraphics[scale=0.3]{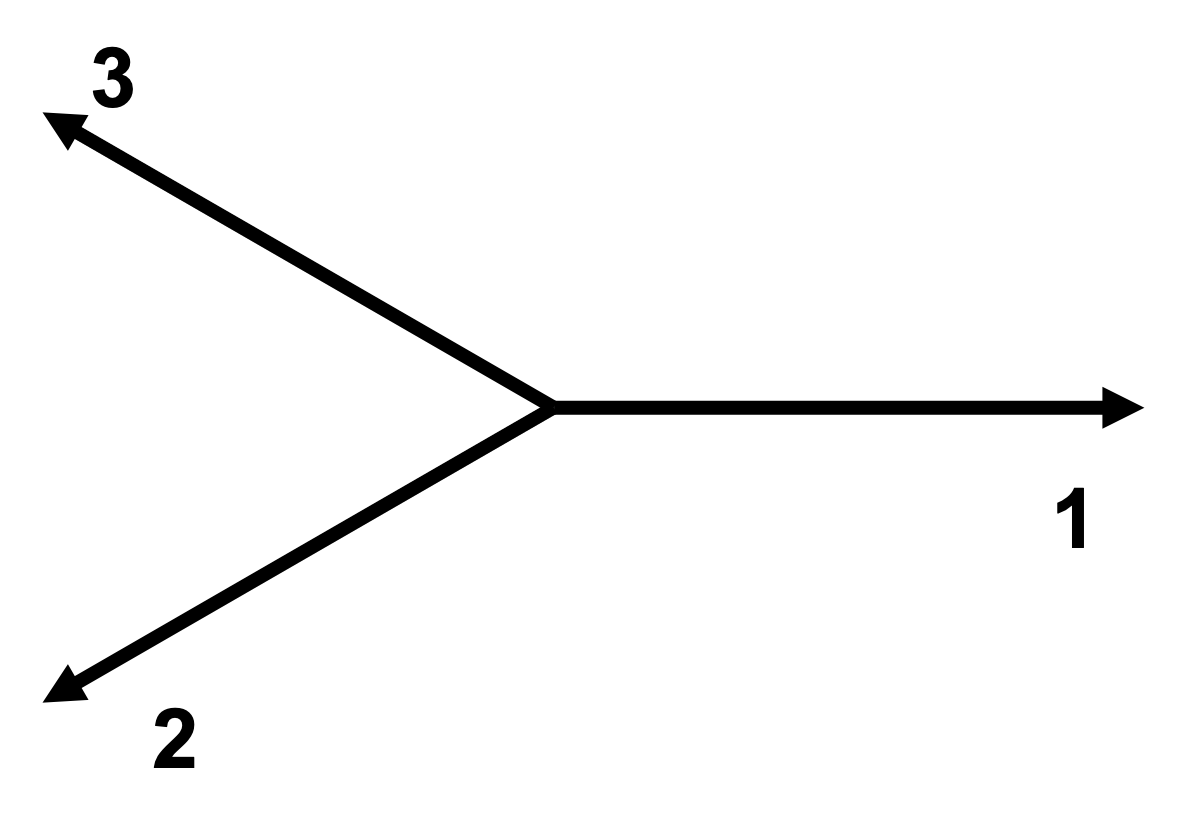}
\end{minipage}
\begin{minipage}{0.3\linewidth}
    \centering
    \includegraphics[scale=0.25]{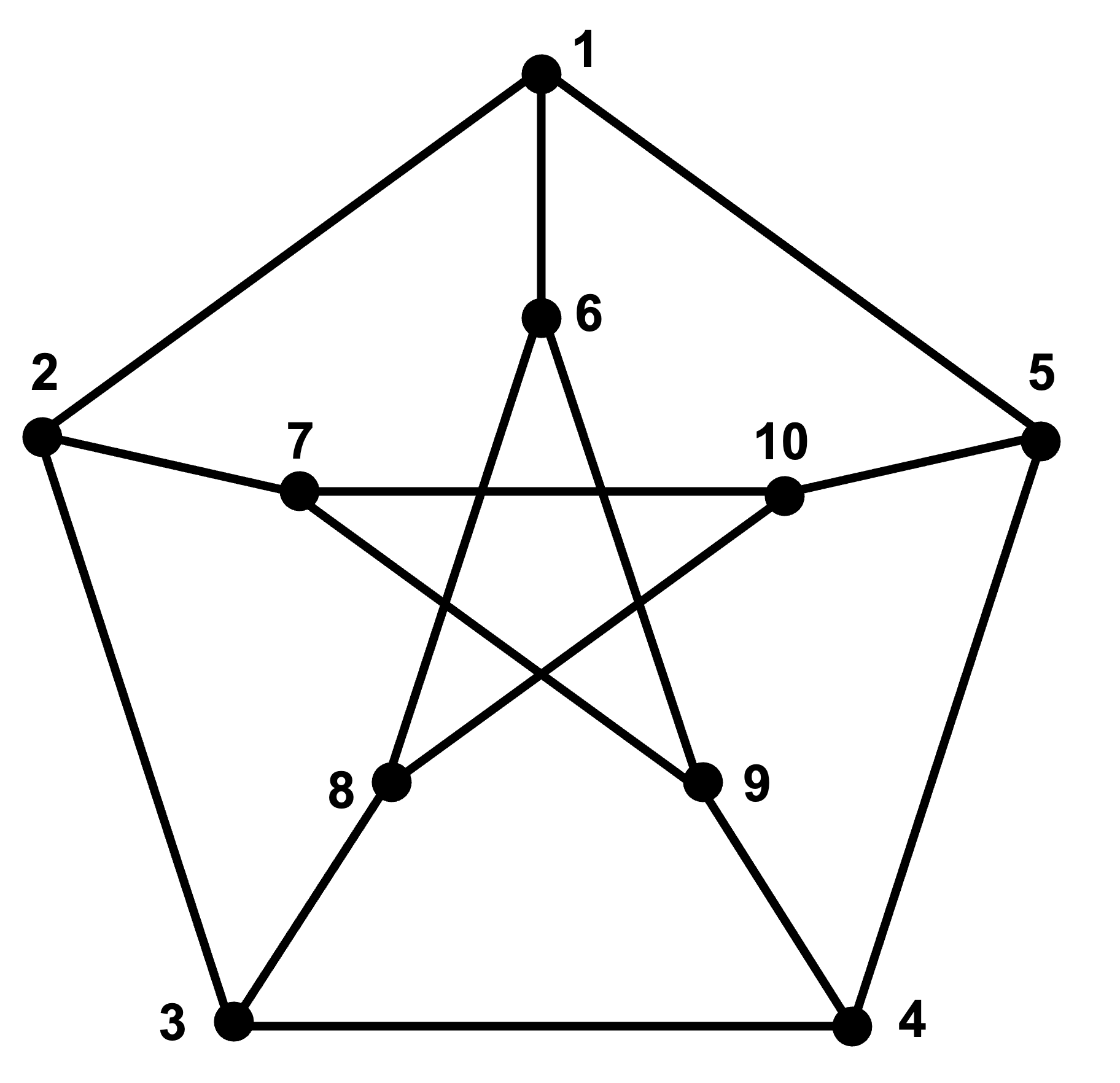}
\end{minipage}
\begin{minipage}{0.3\linewidth}
    \centering
    \includegraphics[scale=0.25]{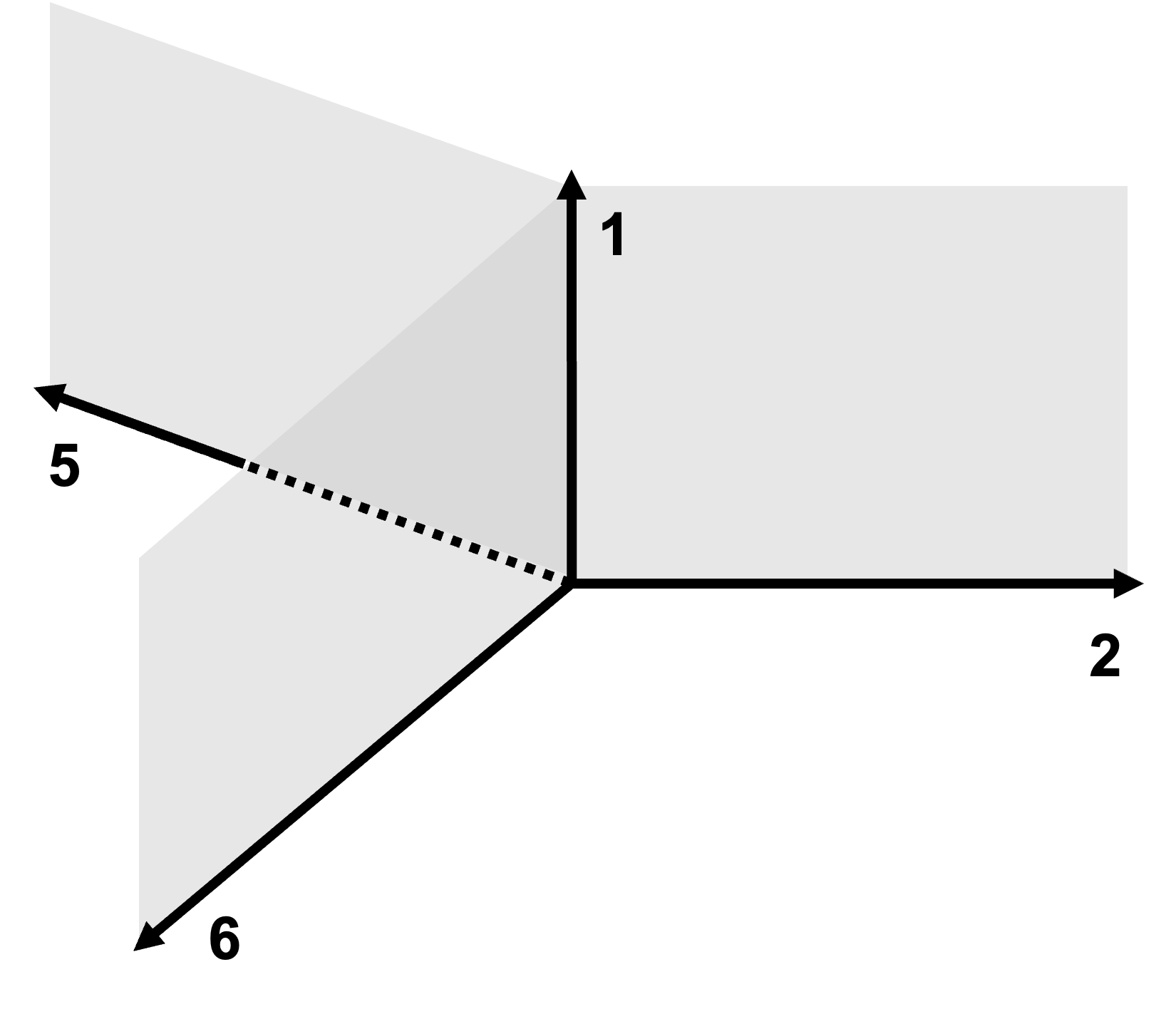}
\end{minipage}
    \caption{(Left): 3-spider. (Center): $\Omega^{(4)}$ corresponding to the phylogenetic tree space $\mathcal{T}_4= \mathcal{O}(\mathcal{E}^{(4)}, \Omega^{(4)})$. Each vertex of the labeled graph corresponds to a singleton set in $\Omega$, and the set of two vertices connected by a line segment corresponds to a two-element set in $\Omega$. (Right): A part of $\mathcal{T}_4$.}
    \label{fig:twodim_treespace}
\end{figure} 

    \subsubsection{CAT(0) space}
    Let $(\mathcal{X}, d)$ be an arbitrary metric space. If any two points in $\mathcal{X}$ can be connected by a path whose length equals the distance between them, such a path is called a geodesic, and $(\mathcal{X}, d)$ is said to be a geodesic metric space. More precisely, a geodesic connecting two points $x, y \in \mathcal{X}$ is a continuous map $\gamma_{x,y}: [0, 1] \to \mathcal{X}$ such that $\gamma_{x,y}(0) = x$, $\gamma_{x,y}(1) = y$, and for any $s, t \in [0, 1]$, we have $d(\gamma_{x,y}(s), \gamma_{x,y}(t)) = |s - t| d(x, y)$.

A CAT(0) space is a geodesic metric space with non-positive curvature. There are multiple equivalent characterizations of CAT(0) spaces; see, for example, \cite{Bacak2014-nn}. In particular, a complete CAT(0) space is called a Hadamard space.

Among the various properties of Hadamard spaces, one of the most fundamental for statistical inference is the uniqueness of geodesics. \cite{Owen2011-nq} developed a polynomial-time algorithm for finding geodesics in phylogenetic tree space. This algorithm was later extended to CAT(0) orthant spaces \citep{Miller2015-jm} and to general CAT(0) cubical complexes \citep{Hayashi2018}.

Furthermore, in Hadamard spaces, convexity can be defined using the unique geodesics. Specifically, a subset $S \subseteq \mathcal{H}$ of a Hadamard space $(\mathcal{H}, d)$ is defined to be convex if the geodesic connecting any two points of $S$ lies entirely within $S$. Additionally, the convexity of a function $f:\mathcal{H} \to [-\infty, \infty]$ is defined via the convexity of its epigraph $\{(x, \mu) \in \mathcal{H} \times \mathbb{R} \mid \mu \geq f(x)\}$. As in the case of Euclidean spaces, the function $f$ is concave if $-f$ is convex. Several facts about convex analysis on Hadamard spaces are summarized in \cite{Bacak2014-nn}. One important fact is the joint convexity of the distance function.

\begin{lemma}[\cite{Bacak2014-nn}]\label{lem:joint_convexity_of_geodesic}
    Let $(\mathcal{H}, d)$ be a Hadamard space, and let $\gamma, \gamma^\prime : [0,1] \to \mathcal{H}$ be two arbitrary geodesics. Then, the distance function satisfies the following joint convexity inequality:
    \begin{align}
        d(\gamma(t), \gamma^\prime(t)) \leq (1-t) d(\gamma(0), \gamma^\prime(0)) + t d(\gamma(1), \gamma^\prime(1)).
    \end{align}
\end{lemma}

In particular, if $\gamma^\prime(t) = y$ for all $t \in [0,1]$ and some $y \in \mathcal{H}$, the inequality reduces to convexity of the distance function with respect to its first (or second) argument.

An important consequence, from a statistical perspective, is the unique existence of the Fréchet mean.

\subsubsection{CAT(0) orthant space}

While orthant spaces are complete geodesic metric spaces, they are not necessarily CAT(0). However, the characterization of CAT(0) orthant spaces can be given in a simple form. Here, $\Omega$ is called a \textit{flag} complex if it satisfies the following condition: for any $F \in 2^\mathcal{E}$, if every pair of elements $e, f \in F$ satisfies ${e,f} \in \Omega$ (that is, every two-point subset of $F$ is an element of $\Omega$), then $F \in \Omega$.

\begin{theorem}[\cite{Gromov1987-vn, Miller2015-jm}]
    A necessary and sufficient condition for an orthant space $\mathcal{O}(\mathcal{E}, \Omega)$ to be CAT(0) is that $\Omega$ is a flag complex.
\end{theorem}


\begin{example}
    The phylogenetic tree space given in Example \ref{ex:treespace} is CAT(0). This is because, when a set $S$ of internal branches of a phylogenetic tree is pairwise compatible, the entire set $S$ is also compatible \citep{Semple2003-ad}.
\end{example}

\subsection{Log-concave density on CAT(0) orthant space}

In this paper, we work with probability density functions, so a reference measure is required. Since CAT(0) orthant spaces are built from Euclidean nonnegative orthants glued along their faces, a natural reference measure can be defined from the usual Lebesgue measure.

Specifically, for a Borel set $S$ in the $p$-dimensional CAT(0) orthant space $\mathcal{O}(\mathcal{E}, \Omega)$, define
\begin{align}
    \nu(S) = \sum_{F \in \Omega_p} \lambda_p(S \cap \bar{O}_F),
\end{align}
where $\lambda_p$ is the Lebesgue measure on $\mathbb{R}^p$. This defines a valid measure. We use the same symbol $\nu$ to denote its completion, which serves as the reference measure. Such a measure has been adopted in several studies on phylogenetic tree spaces \citep{Willis2019-fx, Takazawa2024-ed}.

It is also straightforward to construct probability measures on
$\mathcal{O}(\mathcal{E},\Omega)$. For instance, a probability measure
$Q$ absolutely continuous with respect to $\nu$ can be obtained by
equipping each orthant $\bar{O}_F$ with a nonnegative integrable
function $f_F$ and setting
\[
   Z = \sum_{F \in \Omega_p} \int_{\bar{O}_F} f_F \, d\lambda_p, \qquad
   f(x) = \frac{1}{Z}\sum_{F \in \Omega_p} f_F(x)\,\mathbf{1}_{\{x \in \bar{O}_F\}}.
\]
Then $f$ is a density with respect to $\nu$, satisfying 
$\int dQ = \int f\, d\nu = 1$.
In later results (Theorems~\ref{th:existence} and~\ref{th:continuity}), we assume that the probability measure $Q$ satisfies 
$\int d(x,x_0)\, dQ < \infty$. 
This condition is analogous to the finite first-moment requirement in $\mathbb{R}^p$ and is readily ensured in the above construction, for example, if each $f_F(x)$ decays at least exponentially with the geodesic distance from some fixed point $x_0$, such as $f_F(x) \propto \exp[-d(x,x_0)]$.

A function $f: \mathcal{O}(\mathcal{E},\Omega) \to [-\infty, \infty]$ on a CAT(0) orthant space $\mathcal{O}(\mathcal{E},\Omega)$ is said to be log-concave if $\log f$ is a concave function. In particular, if a log-concave function $f$ is a probability density with respect to $\nu$, i.e. $\int_{\mathcal{O}(\mathcal{E},\Omega)} f , d\nu = 1$, then $f$ is called a log-concave density.

The class of log-concave densities in Euclidean space includes many practical distributions, such as multivariate normal distributions, gamma distributions with shape parameters of 1 or greater, and beta distributions where both shape parameters are 1 or greater. Although this class is non-parametric, \cite{Cule2010-xh} showed that, in a $p$-dimensional Euclidean space $\mathbb{R}^p$, the maximum likelihood estimator exists with probability 1 when the sample size $N$ is greater than or equal to $p+1$. Additionally, they showed that the likelihood maximization reduces to a convex optimization problem of dimension $N$. Similar properties have also been established for BHV tree space \cite{Takazawa2024-ed}, and the same arguments extend to CAT(0) orthant spaces.

In this study, we consider the following class of log-concave functions on the CAT(0) orthant space $\mathcal{O}(\mathcal{E}, \Omega)$. First, let $\Phi$ be the set of concave, upper-semicontinuous, coercive functions on $\mathcal{O}(\mathcal{E}, \Omega)$. Here, a function $\phi$ is said to be coercive if for some (and hence any) $x_0 \in \mathcal{O}(\mathcal{E}, \Omega)$, $\phi(x) \to -\infty$ as $d(x, x_0) \to \infty$. We define $\exp(\Phi) = \{ \exp(\phi) \mid \phi \in \Phi \}$ as the class of log-concave functions. These constraints are not restrictive in the context of the following discussions. Moreover, let $\Phi_0 \subseteq \Phi$ denote the subset of functions that are also log-densities, and we define $\exp(\Phi_0)$ as the space of log-concave densities.

Using the convexity of geodesics (Lemma \ref{lem:joint_convexity_of_geodesic}), we can construct a typical example of a log-concave function.

\begin{proposition}
Let $g: [0, \infty) \to [-\infty, \infty)$ be a nonincreasing concave function. Then, for any $x_0 \in \mathcal{O}(\mathcal{E}, \Omega)$, the function $\phi : \mathcal{O}(\mathcal{E}, \Omega) \to [-\infty, \infty)$ defined by
\[
\phi(x) = g(d(x, x_0)),
\]
is concave.
\end{proposition}

\begin{proof}
By the joint convexity of the distance function, for any $x,y \in \mathcal{O}(\mathcal{E}, \Omega)$ and $\alpha \in (0,1)$, 
\begin{align}
    d(\gamma_{x,y}(\alpha), x_0) \leq (1-\alpha)d(x, x_0) + \alpha d(y, x_0).
\end{align}
The result follows.
\end{proof}

\subsection{Kernel density estimator}
In \cite{Weyenberg2014-wd, Weyenberg2017-kh}, the authors proposed the kernel density estimator on tree space, mainly for the purpose of outlier detection. In this section, we briefly review their contribution.

Given a sample $\{X_1, \ldots, X_N\}$ on tree space $\mathcal{T}_{p+2}$, the estimator takes the form
\begin{align}
   \hat{f}(X) \propto \frac{1}{N}\sum_{i=1}^N K_{X_i,h}(X),
\end{align}
where $K_{X_i,h}$ is a kernel function and $h$ is the bandwidth. In both their papers and the associated \texttt{R} implementation, the kernel was chosen to be of Gaussian type:
\begin{align}
   K_{X_i,h}(X) \propto \exp\!\left( -\left( \frac{d(X,X_i)}{h} \right)^2 \right).
\end{align}

In Euclidean space, the normalizing constant for this kernel function has a closed-form expression:
$c(X, h) = (2\pi h)^{-p/2}$. However, in tree space, no such simple expression is available. In fact, the normalizing constant depends on the location of $X$. This can be seen by comparing two extreme cases: when $X$ is at the origin and when $X$ lies deep within an orthant, far from any boundaries. Although \citet{Weyenberg2014-wd} did not address this issue and instead used a global constant independent of location, \citet{Weyenberg2017-kh} proposed using a lower bound for the normalizing constant, which can be computed using either classical quadrature methods or the holonomic gradient method.

As we show in Section \ref{sec:boundary_bias}, however, this estimator suffers from bias at the boundaries of the nonnegative orthants. In light of this, we propose a modified kernel designed to mitigate this issue.


\section{Theoretical Properties of Log-concave Projection}
In this section, we derive some theoretical properties of the log-concave maximum likelihood estimator. Following \cite{Dumbgen2011-mq} in the Euclidean setting, we work with the more general concept of log-concave projection.

\subsection{Log-concave projections}
\citet{Dumbgen2011-mq} considered the problem of approximating any probability measure on Euclidean space by a log-concave density. They studied the minimization of a Kullback–Leibler divergence–type quantity and called the minimizing log-concave density, if it exists, the log-concave approximation (also known as the log-concave projection). The exact characterization for the existence and other theoretical properties of log-concave projections was derived in the same paper, and some further explorations of the theoretical properties followed (see, for example, \cite{Samworth2018-gf} for a review). The first goal of the present paper is to extend these results, albeit partially, to the setting of non-positively curved orthant spaces. \citet{Takazawa2024-ed} provided a sufficient condition for the existence of the log-concave maximum likelihood estimator of an i.i.d. sample, which corresponds to the log-concave projection of the empirical measure, in BHV tree space. It is straightforward to verify that the existence results presented there also hold in CAT(0) orthant spaces.

As in \cite{Dumbgen2011-mq}, we consider the following functional $L$. For a concave function $\phi \in \Phi$ and a probability measure $Q$ on $\mathcal{O}(\mathcal{E}, \Omega)$, let
\begin{align}
L(\phi,Q) \coloneqq \int \phi \, dQ - \int e^{\phi} \, d\nu + 1.
\end{align}

Note that if $\int e^{\phi} \, d\nu = 1$ and $Q$ has a density $f_Q$ satisfying $\int f_Q \lvert \log f_Q \rvert \, d\nu < \infty$, then maximizing $L(\phi, Q)$ is equivalent to minimizing the Kullback–Leibler divergence from $Q$ to the density $\exp(\phi)$. However, here $L$ is defined for all concave functions $\phi \in \Phi$, not only for those corresponding to log-densities. 

Let $L(Q)$ denote the supremum of $L(\phi, Q)$ over all concave functions in $\Phi$. We call any $\psi \in \Phi$ attaining $L(\psi,Q)=L(Q)$, or the corresponding density $\exp(\psi)$, the \emph{log-concave projection} of $Q$.
The next lemma states that any maximizer of $L$ with respect to the first argument, if it exists, is always a log-density. Therefore, we may consider maximization of $L$ in the space $\Phi$, for finding log-concave projections.

\begin{lemma} \label{lem:maximizer-is-density}
    Fix probability measure $Q$ on $\mathcal{O}(\mathcal{E}, \Omega)$. Assume that $L(Q)\in\mathbb{R}$ and that there exists $\psi\in\Phi$ that achieves $L(Q)$:
    \begin{align}
    L(\psi,Q) = L(Q)\coloneqq\sup_{\phi\in\Phi}L(\phi,Q).
    \end{align}
    Then, $\psi$ is a log-density:
    \begin{align}\int e^{\psi} d\nu = 1.\end{align}
\end{lemma}

\begin{proof}
    For any $\phi\in\Phi$, we have $\int e^\phi d\nu<\infty$ because of coercivity. Also, we only need to consider the case $\int e^\phi d\nu > 0$, since if it equals zero, then $L(Q)$ cannot be finite. To see this, assume otherwise. Then, for any $M>0$, $\phi + M \in \Phi$ and $L(\phi + M, Q) = L(\phi, Q) + M$. This would imply that $L(Q) = \infty$, a contradiction. Now, let $0<c = \int e^\phi d\nu$. Then, $\phi-\log c\in\Phi$, $\int e^{\phi-\log c}d\nu=1$, and
    \begin{align}L(\phi-\log c, Q) \nonumber
    &= \int (\phi-\log c)dQ - \int e^{\phi-\log c}d\nu + 1 \nonumber \\
    &= \int \phi dQ - \log c \nonumber\\
    &= L(\phi,Q) + c - 1 - \log c \nonumber\\
    &\geq  L(\phi,Q).
    \end{align}
    The equality holds if and only if $c=1$.
\end{proof}

Throughout this section, we consider a probability measure $Q$ that has a density $q$.

\subsection{The uniqueness and existence of log-concave projections}
    We first show some properties of the log-concave projection. For any probability measure $Q$, let $\exp(\psi_Q)$ denote its log-concave projection. Also, the convex support of a probability measure $Q$, denoted by $\csupp(Q)$, is defined as follows:
    \begin{align}
        \csupp(Q) = \cap \{C\subseteq \mathcal{O}(\mathcal{E},\Omega)\mid C \text{ is closed and convex, and }Q(C)=1 \}.
    \end{align}
    $\csupp(Q)$ is the smallest closed convex set with $Q$-measure 1.

    We begin with the following Lemma.
    \begin{lemma}\label{lem:Q_h}
    Suppose that $Q$ has a density. Let 
    \begin{align}
        h(Q,x) = \sup\{ Q(C) \mid C : \text{closed and convex }, x\not \in \inter(C) \}. \label{eq:def_of_h}
    \end{align}
     Then, if $x\in \inter({\csupp(Q)})\cap \cup_{F\in\Omega_p}O_F$, we have $h(Q,x) < 1$.
    \end{lemma}
    The proof is relegated to Appendix \ref{sec:proof-logcon}.
    
    The uniqueness of the log-concave projection holds $\nu$-almost everywhere, whenever it exists.

    \begin{theorem}\label{th:uniqueness}
        Log-concave projection $\exp(\psi_{Q})$ of any probability measure $Q$ that has a density on a CAT(0) orthant space is unique $\nu$-almost everywhere. Furthermore, it is strictly unique on $\inter(\csupp(Q))\cap \cup_{F\in\Omega_p} O_F$.
    \end{theorem}

    \begin{proof}
        Suppose that two concave functions $\psi_1, \psi_2\in\Phi$ both maximize $L(\cdot,Q)$. Then, by Lemma \ref{lem:maximizer-is-density}, $\int\exp(\psi_1)d\nu = \int \exp(\psi_2)d\nu = 1$. Hence, if we let $\psi = \psi_1/2 + \psi_2/2 - \log \int \exp((\psi_1 + \psi_2)/2)d\nu$,
        \begin{align}
            L(\psi, Q) &= \frac{1}{2}(L(Q) + L(Q)) - \log\int \exp\left(\frac{\psi_1 + \psi_2}{2}\right)d\nu \nonumber \\
            &\geq L(Q) - \log\int\frac{\exp(\psi_1) + \exp(\psi_2)}{2}d\nu \nonumber \\
            &= L(Q) 
        \end{align}
        Equality holds if and only if $\exp(\psi_1) = \exp(\psi_2)$ $\nu$-almost everywhere.

        Furthermore, for any log-concave projection $\psi_Q$, we have 
        \[\dom(\psi_Q) \supseteq \inter(\csupp(Q)) \cap \cup_{F\in\Omega_p} O_F.\]
        To see this, let $M = \sup_x \psi_Q(x)$.
        For an interior point $x\in\inter(\csupp(Q))\cap \cup_F O_F$, if $\psi_Q(x)< M$, 
        \begin{align}
            L(\psi_Q, Q) &= \int \psi_Q dQ \nonumber \\
            &= \int_{y: \psi_Q(y) \leq \psi_Q(x)} \psi_Q(y) dQ(y) + \int_{y: \psi_Q(y) > \psi_Q(x)} \psi_Q(y) dQ(y)  \nonumber\\
            &\leq \psi_Q(x) + (M - \psi_Q(x))h(Q,x)
        \end{align}
        where $h$ is the function defined in Lemma \ref{lem:Q_h}. Note that this relationship holds even when $\psi_Q(x) = M$. Then, 
        \begin{align}
            \psi_Q(x) \geq \frac{L(Q) - Mh(Q,x)}{1-h(Q,x)}
        \end{align}
        This shows that $\psi_Q(x) > -\infty$; thus $x\in\dom\psi_Q$. 

        This in particular implies that $\inter(\dom\psi_Q) \supseteq \inter(\csupp(Q))\cap \cup_{F\in\Omega_p} O_F$. Moreover, the continuity of $\psi_Q$ on $\inter(\dom\psi_Q) \cap \cup_F O_F$ and its $\nu$-a.e. uniqueness imply that the values at $\inter(\csupp(Q)) \cap \cup_F O_F$ are uniquely determined.
    \end{proof}

    \begin{remark}
        A geodesic metric space is said to have the \textit{geodesic extension property} if any geodesic between two points can be extended to form a geodesic line. If $\mathcal{O}(\mathcal{E},\Omega)$ has the geodesic extension property, the uniqueness holds on $\inter(\csupp(Q))$. This follows because, by Lemma 9 of \cite{Takazawa2024-ed}, $\psi_Q$ is continuous on $\inter(\dom\psi_Q)$, which implies that $\psi_Q(x)$ is uniquely determined.
    \end{remark}

    On the other hand, it is straightforward to see that for any log-concave function $\phi\in\Phi$, its restriction to $\csupp(Q)$, denoted $\phi_{\mid\csupp(Q)}$, has a value of $L(\cdot, Q)$ that is greater than or equal to that of $\phi$. Thus, if a log-concave projection exists, we can choose one whose domain is contained in $\csupp(Q)$. Furthermore, because the boundary of a convex set has $\nu$-measure zero, the value of $L(Q, \psi)$ depends on $\psi$ only through its values on $\inter(\csupp(Q)) \cap \cup_F O_F$, when $Q$ has a density and the domain of $\psi$ is contained in $\csupp(Q)$. From here on, we refer to the smallest version of these as the log-concave projection, obtained by taking the pointwise minimum of all log-concave projections.
    
    The final part of this subsection provides a sufficient condition for the existence of log-concave projections on a CAT(0) orthant space.

\begin{theorem}[Existence of log-concave projections] \label{th:existence}
        Assume that the probability measure $Q$ on the CAT(0) orthant space $\mathcal{O}(\mathcal{E}, \Omega)$ satisfies the following conditions:
        \begin{itemize}
            \item $Q$ has a density with respect to $\nu$ and
            \item There exists $x_0\in \mathcal{O}(\mathcal{E},\Omega)$ such that $\int d(x,x_0)dQ < \infty$.
        \end{itemize}
        Then $L(Q)\in\mathbb{R}$. When this holds, there exists 
        \begin{align}\psi \in\argmax_{\phi\in\Phi} L(\phi,Q)\end{align}
        and it is unique $\nu$-almost everywhere. Furthermore, the smallest version of the log-concave projection $\psi$ satisfies the following:
    \begin{align}&\int e^\psi d\nu=1, ~~~\mathrm{int}(\csupp(Q)) \cap \cup_F O_F \subseteq \dom(\psi)\subseteq \csupp(Q).\end{align}
\end{theorem}

The proof is relegated to Appendix \ref{sec:proof-logcon}.

\begin{remark}
    As before, if the space $\mathcal{O}(\mathcal{E}, \Omega)$ has the geodesic extension property, the domain of $\psi$ must contain $\mathrm{int}(\csupp(Q))$ by concavity.
\end{remark}

\subsection{Continuity of log-concave projection map}
In order to prove the consistency of the log-concave MLE, we first show the following continuity result of the log-concave projection map.

\begin{theorem}[Continuity of $L$]\label{th:continuity}
    Assume that the probability measures $(Q_N)_N$ and $Q$ satisfy the following conditions:
    \begin{itemize}
        \item $(Q_N)_N \rightarrow Q$ in the Wasserstein-1 distance.
        \item For some $x_0$, $\int d(x, x_0) dQ_N < \infty$ for all $N$, and $\int d(x, x_0) dQ < \infty$.
        \item $Q$ has a density with respect to $\nu$.
        \item There exists $N_0 \in \mathbb{N}$ such that, for all $N \geq N_0$, the maximizer of $L(\phi, Q_N)$ exists.
    \end{itemize}
    Then, $\lambda \coloneqq \lim_{N\to\infty}L(Q_N) = L(Q)$.

    Furthermore, $\lim_{N\to\infty}\int|\exp(\psi_{Q_N})-\exp(\psi_Q)|d\nu = 0$.
\end{theorem}
The proof can be found in Appendix \ref{sec:proof-logcon}.

\begin{remark}
   Note that convergence in the Wasserstein-1 distance is equivalent to satisfying the following two conditions (e.g., \cite{Villani2016-vm}, Definition 6.8 and Theorem 6.9):
    \begin{itemize}
        \item $(Q_n)_n \xrightarrow{d} Q$
        \item There exists $x_0$ such that $\int d(x, x_0) dQ_n \rightarrow \int d(x, x_0) dQ$.
    \end{itemize}
    In particular, the second condition holds for arbitrary $x_0$.
\end{remark}

As a corollary to this theorem, we establish a condition for the consistency of the log-concave maximum likelihood estimator.
\begin{corollary}[Consistency of maximum likelihood estimator] \label{cor:consistency_lcmle}
Suppose $\{X_N\}_N$ is an i.i.d. sample from a probability measure $P$ that has a density with respect to $\nu$ and satisfies $\int d(x, x_0) dP < \infty$. Let $\mathbb{P}_N$ denote the empirical measure.
        
        Then, for any $\epsilon > 0$, there exists $N_0 \in \mathbb{N}$ such that if $N \geq N_0$, with probability greater than $1 - \epsilon$, the maximum likelihood estimator $\exp(\psi_{\mathbb{P}_N})$ for the sample ${X_1, \ldots, X_N}$ exists and 
        \begin{align}
         \lim_{N\to\infty}\int |\exp(\psi_{\mathbb{P}_N})(x) - \exp(\psi_P)(x)| d\nu = 0.
        \end{align}
        
        In other words, the maximum likelihood estimator converges to the log-concave projection in the total variation distance.
\end{corollary}

\begin{proof}
    By Varadarajan’s theorem \cite{Varadarajan1958-co}, the empirical measure converges weakly to the true measure. By the law of large numbers, $\int d(x,x_0)d\mathbb{P}_N \to \int d(x,x_0) dP$ with probability 1. Together, these results imply that the empirical measure converges to $P$ in the Wasserstein-1 distance.
\end{proof}



The following result describes some properties of the weak convergence of log-concave densities. This result generalizes Proposition 2 of \cite{Cule2010-uy}.

\begin{proposition}\label{prop: weak_conv_log_conc}
Let $f_N$ be a sequence of log-concave densities in $\exp(\Phi)$ that weakly converges to some density $f$. Then,
    \begin{enumerate}
        \item $f$ has a log-concave version ($\exists g$: log-concave such that $g=f~\mathrm{a.e.}$)
        \item $f_N \to f$ $\nu$-almost everywhere
        \item Let $a_0>0$ and $b_0\in\mathbb{R}$ be such that $f(x)\leq \exp(-a_0 d(x,x_0) + b_0)$. Assume also that $\dom \log f_N \subseteq \dom \log f$. Then for every $a<a_0$, we have $\int e^{a d(x,x_0)}|f_N(x)-f(x)|d\nu(x)\to 0$.  Furthermore, if $f$  is continuous and $\mathcal{O}(\mathcal{E},\Omega)$ has the geodesic extension property, uniform convergence also holds:
\begin{align}
    \sup_{x\in\mathcal{O}(\mathcal{E},\Omega)} e^{ad(x,x_0)}|f_N(x)-f(x)|\to 0.
\end{align}
    \end{enumerate}
\end{proposition}
The proof is relegated to Appendix \ref{sec:proof-logcon}.

This proposition implies that the convergence statement in Corollary \ref{cor:consistency_lcmle} can be strengthened as follows:: For $a>0$,  $b\in\mathbb{R}$ and $x_0\in \mathcal{O}(\mathcal{E},\Omega)$ that satisfies $\psi_{P}\leq -ad(x,x_0)+b$, 
\begin{align}
    \lim_{n\to\infty}\int \exp(ad(x,x_0))|\exp(\psi_{P_n})(x) - \exp(\psi_P)(x)| d\nu = 0.
\end{align}

\begin{remark}[Product space setting]
In phylogenetic applications, it is common to consider not only internal branch lengths (modeled in BHV tree space) but also external branch lengths. This leads naturally to the product space $\mathcal{T}_n \times \mathbb{R}^m$, combining a CAT(0) orthant space with a Euclidean space. Let $\mathcal{X}$ be a CAT(0) orthant space with metric $d_\mathcal{X}$ and reference measure $\nu_\mathcal{X}$, and consider the product space $\mathcal{X}\times\mathbb{R}^m$ equipped with the product metric $d\big((x,u),(x',u')\big)=\sqrt{d_\mathcal{X}(x,x')^2+\|u-u'\|^2}$ and the measure $\nu=\nu_\mathcal{X}\otimes\lambda_m$. All results in this section remain valid on $\mathcal{X} \times \mathbb{R}^m$ when $d$ and $\nu$ are replaced accordingly.
\end{remark}

\subsection{Computation and visualization of the log-concave MLE}
The practical use of the log-concave maximum likelihood estimator is limited by computational limitations. 
The support of the log-concave MLE coincides with the convex hull of the sample points in the CAT(0) orthant space $\mathcal{O}(\mathcal{E}, \Omega)$. 
In the optimization algorithm proposed in \cite{Takazawa2024-ed}, optimization is performed over an $N$-dimensional vector associated with 
the sample $\{X_i \in \mathcal{O}(\mathcal{E}, \Omega)\}_{i=1}^N$. At each iteration, the convex hull of the lifted points $\{(X_i,y_i)\}$ must be computed, where the additional coordinates $y_i$ correspond to optimization variables. 
This requires constructing convex hulls in the product of tree space and $\mathbb{R}$, a task that rapidly becomes intractable as the number of taxa increases.
Consequently, implementations are currently limited to one- and two-dimensional cases, and it remains an open problem to develop algorithms that either avoid explicit convex hull computation or yield efficient approximations.

To illustrate the behavior of the estimator, we plot the estimated density alongside the true density for two examples in the 3-spider space. The two densities are defined as follows:
\begin{align}
    f_1(x) &\propto \exp\left(- \tfrac{d(x,x_0)^2}{2 \cdot 0.5^2} \right), \label{eq:f1} \\
    f_2(x) &\propto 0.5 \exp\left(- \tfrac{d(x,x_1)^2}{2 \cdot 0.4^2} \right) 
            + 0.5 \exp\left(- \tfrac{d(x,x_2)^2}{2 \cdot 0.4^2} \right), \label{eq:f2}
\end{align}
where $x_0$ is the point at coordinate $0.5$ on orthant 0, $x_1$ is the point at coordinate $0.6$ on orthant 0, and $x_2$ is the point at coordinate $0.6$ on orthant 1.
From a sample of size $N=1000$, we compute the log-concave MLE and compare it with the true density. 
Figure~\ref{fig:logconcave-viz} presents the estimated densities across the three orthants. 
The support of the estimated density is bounded, reflecting that its support coincides with the convex hull of the sample points.
In the mixture case, the estimator approximates the log-concave projection of the true density, resulting in the smoothing of the two local modes.

\begin{figure}[t]
    \centering
    \begin{subfigure}{0.95\textwidth}
        \centering
        \includegraphics[width=\linewidth]{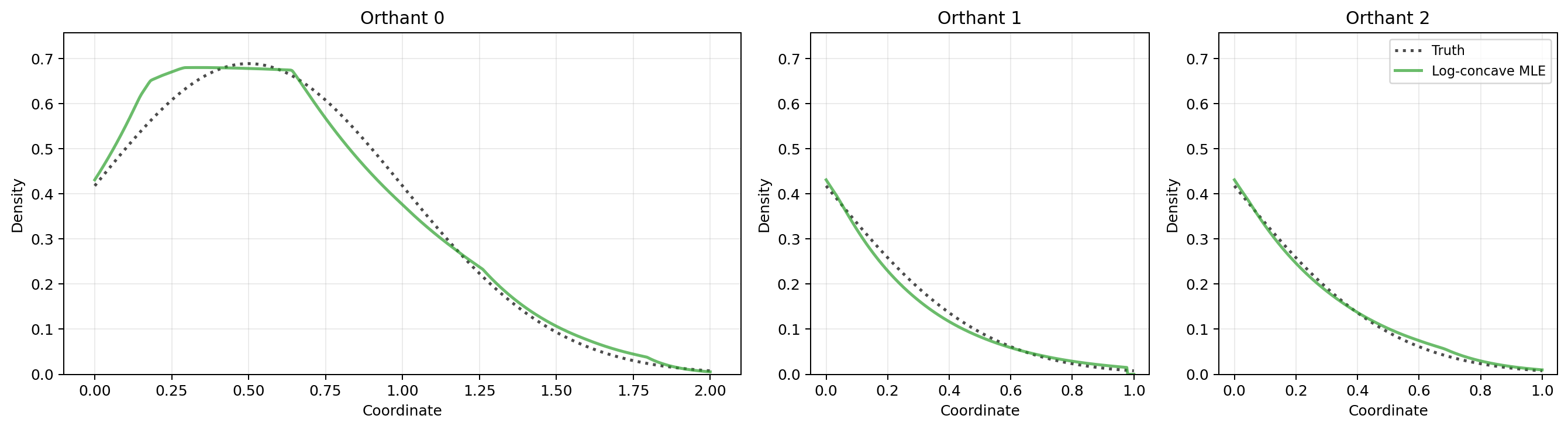}
        \caption{Gaussian-type density $f_1$.}
        \label{fig:logconcave-f1}
    \end{subfigure}
    \vskip\baselineskip
    \begin{subfigure}{0.95\textwidth}
        \centering
        \includegraphics[width=\linewidth]{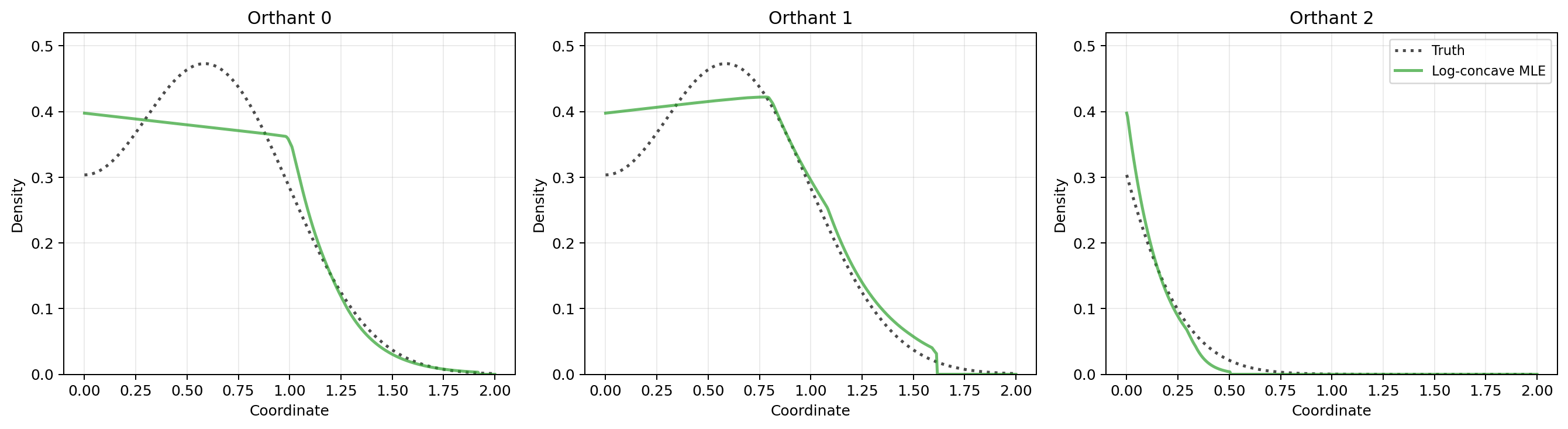}
        \caption{Gaussian-type mixture density $f_2$.}
        \label{fig:logconcave-f2}
    \end{subfigure}
    \caption{Log-concave MLE in the 3-spider space. Each subfigure displays the three orthants, overlaying the true density (black dotted) and the log-concave MLE (green solid). Panel (a) corresponds to the Gaussian-type density $f_1$, and panel (b) to the mixture density $f_2$.}
    \label{fig:logconcave-viz}
\end{figure}

\section{Consistency of Kernel Density Estimator}
\label{sec:kde}
In this section, we first demonstrate that the kernel density estimator proposed in \cite{Weyenberg2014-wd, Weyenberg2017-kh} exhibits boundary bias. We then introduce a slightly modified version of the estimator and establish a uniform consistency result.

\subsection{Kernel density estimator and boundary bias}
\label{sec:boundary_bias}
Let $X_1, X_2, \ldots, X_N$ be independent and identically distributed random variables taking values in the orthant space $\mathcal{O}(\mathcal{E},\Omega)$, drawn from a probability distribution whose density function $f$ is uniformly continuous and bounded on $\mathcal{O}(\mathcal{E},\Omega)$. We estimate $f$ using a kernel density estimator defined as:
\begin{align}
    \hat{f}_{N,h}(x) = \frac{1}{N} \sum_{i=1}^N K_{X_i, h}(x),
\end{align}
where $K_{X_i,h}$ is the kernel function. We do not require $K_{X_i,h}$ to be either symmetric or a density (i.e., to integrate to 1).

\cite{Weyenberg2014-wd, Weyenberg2017-kh} considered the following type of kernel function:
\begin{align}
    K_1(x\mid X_i, h) \coloneqq K_{X_i,h}(x) = C(X_i,h) \exp\left( -\frac{d(x,X_i)^2}{2h^2} \right),
\end{align}
where $C(X_i,h)$ denotes the normalization constant, defined as 
\begin{align}
     C(X_i,h) =\left(\int \exp\left( -\frac{d(x,X_i)^2}{2h^2} \right) d\nu(x) \right)^{-1}. \label{eq:normalizing-constant}
\end{align}
Note that, as mentioned in Section \ref{sec:preliminaries}, the normalizing constant \eqref{eq:normalizing-constant} is location-dependent.

Although the definition of the kernel appears natural, this density generally exhibits an asymptotic bias at the boundary of the orthant space, similar to the case of bounded densities in Euclidean space. The following is a simple example on $k$-spider.

\begin{example}\label{ex:boundary_bias}
    Consider estimating a bounded continuous density $f$ on a $k$-spider ($k \geq 3$). Let $|x|$ denote the coordinate of $x$ on the half-line. The kernel function in this case is:
    \begin{align}
        K_1(x\mid X_i, h) = \frac{1}{\sqrt{2\pi}h (1+(k-2)\Phi(-\|X_i\|/h))}\exp\left(-\frac{d(x,X_i)^2}{2h^2}\right),
    \end{align}
    where $\Phi(u)$ is the cumulative distribution function of standard normal on $\mathbb{R}$. 

    Now, define the kernel-smoothed density $f_h$ by
    \begin{align}
        f_h(y) = \mathbb{E}_X[K_1(y\mid X, h)] = \int K_1(y\mid x, h)f(x) d\nu(x).
    \end{align}
    It can be shown that $f_h(0)\not\to f(0)$ as $h\to 0$ when $f(0)>0$. This implies the asymptotic bias of the kernel density estimator at the origin.

    In fact, 
    \begin{align}
        f_h(0) &= \int \frac{1}{\sqrt{2\pi}h (1+(k-2)\Phi(-\|x\|/h))}\exp\left(-\frac{\|x\|^2}{2h^2}\right) f(x) d\nu(x), \nonumber\\
        &= \sum_{l=1}^k \int_{x\in O_l} \frac{1}{\sqrt{2\pi}h (1+(k-2)\Phi(-\|x\|/h))}\exp\left(-\frac{\|x\|^2}{2h^2}\right) f(x) d\nu(x), \label{eq:integ_all_ort}
    \end{align}
    where $O_l$ represents the $l$-th half-line.
    Let each orthant's contribution (each integrand of \eqref{eq:integ_all_ort}) $I_l(h)$, and restriction of $f$ to $l$-th orthant $f_l$. Then, 
    \begin{align}
        I_l(h) &= \int_0^\infty \frac{1}{\sqrt{2\pi}h(1+(k-2)\Phi(-x/h))}\exp\left(-\frac{x^2}{2h^2}\right) f_l(x) dx \nonumber\\
        &= \int_0^\infty \frac{\phi(y)}{1+(k-2)\Phi(-y)} f_l(hy) dy \nonumber\\
        &= f(0)\int_0^\infty \frac{\phi(y)}{1+(k-2)\Phi(-y)}dy \nonumber \nonumber \\
        &+  \int_0^\infty \frac{\phi(y)}{1+(k-2)\Phi(-y)} (f_l(hy)-f(0)) dy. \label{eq:il}
    \end{align}
    Here, $\phi(y)$ denotes the density function of the standard normal distribution.

    Because of the boundedness and continuity of $f$, the second term vanishes as $h\to 0$. Then, as $h\to 0$,
    \begin{align}
        I_l(h) \to f(0)\int_0^\infty \frac{\phi(y)}{1+(k-2)\Phi(-y)}dy = \frac{f(0)}{k-2}\log\left(\frac{k}{2}\right). \label{eq:il_limit}
    \end{align}
    Consequently, as $h\to 0$,
    \begin{align}
        f_h(0)\to f(0)\frac{k}{k-2}\log\left(\frac{k}{2}\right) > f(0).
    \end{align}
\end{example}

Example \ref{ex:boundary_bias} shows that the bias at the origin arises because the following term, corresponding to the sum of the first terms in the decomposition of \eqref{eq:il}, does not equal one:
\begin{align}
    \int_{\mathcal{O}(\mathcal{E}, \Omega)} C(y,h)\, \exp\left(-\frac{d(0,y)^2}{2h^2}\right) d\nu(y) > 1. \label{eq:bias-general}
\end{align}
This inequality holds because the normalizing constant $C(y,h)$ attains its minimum at the origin. In contrast, if we replace $C(y,h)$ by $C(0,h)$ in \eqref{eq:bias-general}, the integral evaluates to one by definition. Thus, to eliminate this form of boundary bias, we consider the following modified kernel:
\begin{align}
    K_2(x\mid X_i, h) = K_{x, h}(X_i) = C(x,h)\exp\left(-\frac{d(X_i,x)^2}{2h^2}\right).
\end{align}
That is, we switch the roles of $x$ and $X_i$ in the definition of the kernel and use the normalizing constant at $x$ rather than at $X_i$. With this choice, the expectation of the kernel-smoothed density becomes a local weighted average:
\begin{align}
    \mathbb{E}_X\!\left[ K_2(x \mid X, h) \right]
    &= \int C(x,h)\,\exp\!\left(-\frac{d(x,y)^2}{2h^2}\right) f(y)\, d\nu(y) \nonumber \\
    &= \int w_{x,h}(y)\, f(y)\, d\nu(y),
\end{align}
where $w_{x,h}(y):=C(x,h)\exp\!\left(-\tfrac{d(x,y)^2}{2h^2}\right)$ is a kernel weight that integrates to one by construction.
Similar approaches can be found in the literature on the estimation of bounded densities in Euclidean space (e.g., \cite{Chen2000-qd, Chen2000-sv}). Note that the kernel density estimator defined in this way does not necessarily integrate to one. However, as we will see, this definition succeeds in avoiding boundary bias.

\subsection{Consistency of the kernel density estimator}

Our main result is the following theorem, which establishes the consistency of the kernel density estimator $\hat{f}_{N,h}$ on $\mathcal{O}(\mathcal{E},\Omega)$ under suitable conditions.

\begin{theorem}\label{th:kde_consistency}
Let $X_1, X_2, \ldots, X_N$ be independent and identically distributed random variables from a uniformly continuous and bounded density function $f$ on $\mathcal{O}(\mathcal{E},\Omega)$. Let $h_N$ be a sequence of bandwidths satisfying $h_N \to 0$ as $N\to \infty$, along with the conditions:

\begin{enumerate}
    \item $\frac{N h_N^p}{|\log h_N|} \to \infty$,
    \item $\frac{|\log h_N|}{\log\log N} \to \infty$,
    \item There exists a constant $\check{c} > 0$ such that $h_N^p \leq \check{c} h_{2N}^p$.
\end{enumerate}

Then, for kernel $K_2(x\mid X_i,h)$, we have:

\begin{equation}
    \sup_{x \in \mathcal{O}(\mathcal{E},\Omega)} |\hat{f}_{N,h_N}(x) - f(x)| \to 0 \quad \text{almost surely as } N \to \infty.
\end{equation}
\end{theorem}

The proof follows a similar argument to that in \cite{Gine2001-be} and is presented in Appendix \ref{sec:proof-kernel-consistency}. The bandwidth conditions in Theorem~\ref{th:kde_consistency} coincide with those imposed in the Euclidean case in \cite{Gine2001-be}.


\begin{remark}[Product space setting]
The uniform consistency established in Theorem~\ref{th:kde_consistency} also extends to product spaces of the form $\mathcal{O}(\mathcal{E}, \Omega) \times \mathbb{R}^m$, equipped with the product metric. In this setting, the dimension parameter $p$ appearing in the bandwidth conditions is replaced by $p+m$.
\end{remark}

\subsection{Computation and visualization of kernel density estimation}

Both the original estimator with kernel $K_1$ and the modified estimator with kernel $K_2$ require evaluating 
location-dependent normalizing constants of the form \eqref{eq:normalizing-constant}. 
Exact computation is infeasible in general, and prior work \cite{Weyenberg2017-kh} addressed this 
by replacing the integral with more tractable approximations. 
At a practical level, $K_1$ centers kernels at the sample points and therefore requires approximations of $N$ normalizing constants for a sample of size $N$,
while $K_2$ centers kernels at the evaluation points and consequently requires as many approximations as the number of query locations.

To illustrate the behavior of these estimators, we consider the Gaussian-type density $f_1$ defined in \eqref{eq:f1} and the two-component mixture density 
$f_2$ defined in \eqref{eq:f2} in the 3-spider space. 
In the 3-spider space, the normalizing constants of the kernels $K_1$ and $K_2$ can be computed using the cumulative distribution function of the standard normal distribution on $\mathbb{R}$.
To evaluate the asymptotic bias, we generate 100 independent samples of size $N=100{,}000$ and assess the performance of both the original KDE estimator and the modified KDE estimator.
Figure~\ref{fig:kde-viz} shows the estimated and true densities across the three orthants for a sample of size $N = 100{,}000$.
The original estimator $K_1$ exhibits upward bias at the origin, 
while the modified estimator $K_2$ shows a smaller downward bias.
This pattern is consistent over replicates, where the mean bias of the original KDE is $0.061$, and the mean bias of the 
modified KDE is $-0.022$. 
These results illustrate that the modified kernel achieves the intended asymptotic correction, although a finite-sample bias remains visible near the boundary.

\begin{figure}[t]
    \centering
    \begin{subfigure}{0.95\textwidth}
        \centering
        \includegraphics[width=\linewidth]{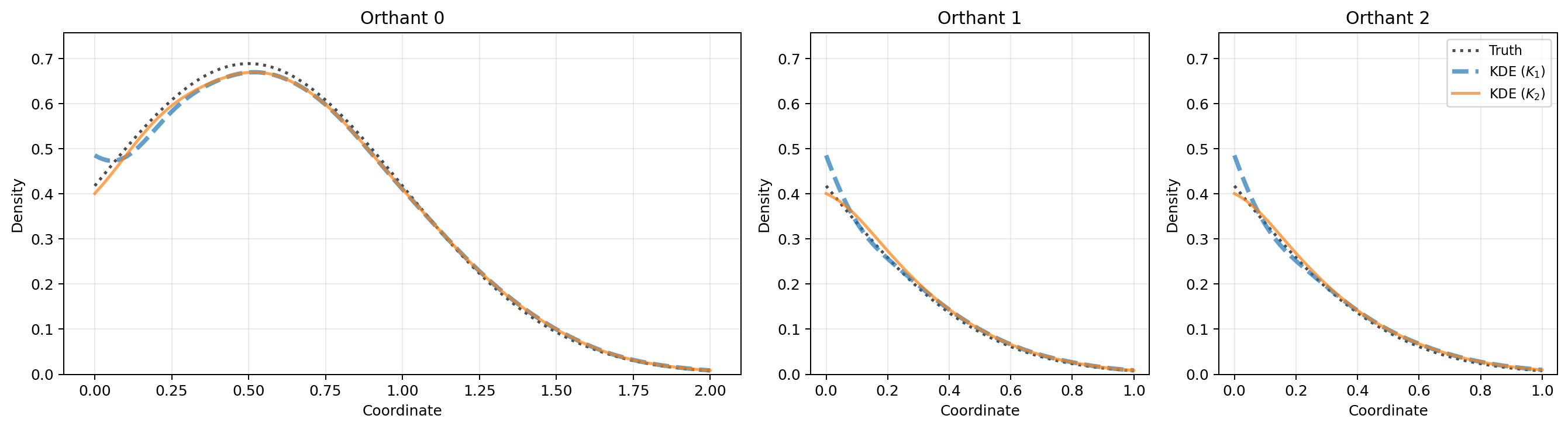}
        \caption{Gaussian-type density $f_1$.}
        \label{fig:kde-f1}
    \end{subfigure}
    \vskip\baselineskip
    \begin{subfigure}{0.95\textwidth}
        \centering
        \includegraphics[width=\linewidth]{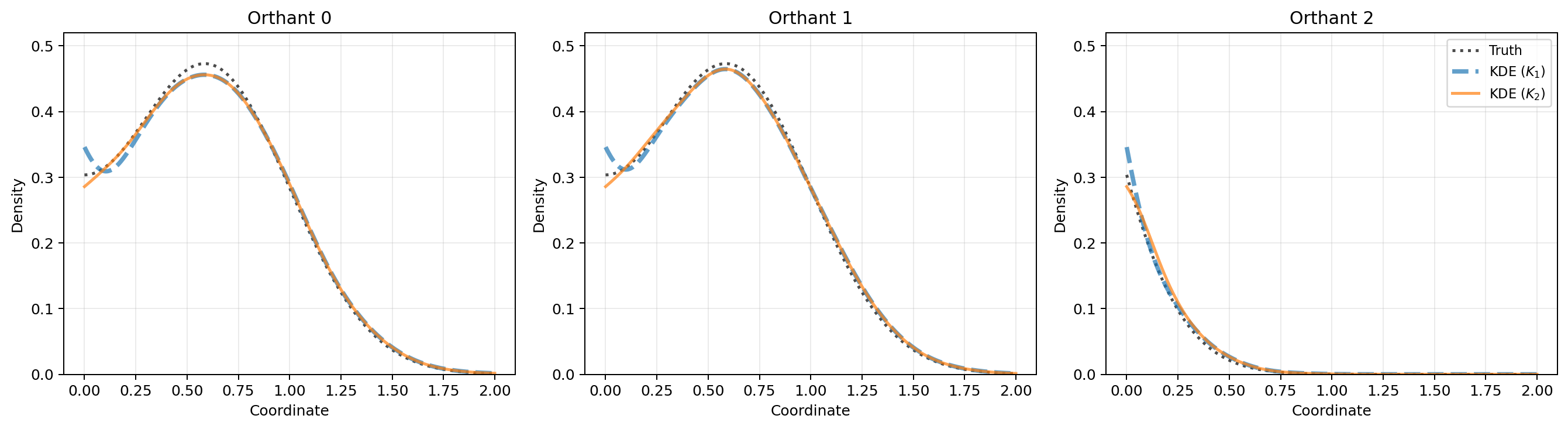}
        \caption{Mixture density $f_2$.}
        \label{fig:kde-f2}
    \end{subfigure}
    \caption{Kernel density estimation in the 3-spider space. 
    Each subfigure shows the three orthants with the true density (black dotted), 
    the original estimator with kernel $K_1$ (blue dashed), and the modified estimator with kernel $K_2$ (orange solid). 
    Panel (a) corresponds to $f_1$, and panel (b) to $f_2$.}
    \label{fig:kde-viz}
\end{figure}

\section{Conclusion}

In this paper, we examined key theoretical properties of the log-concave maximum likelihood estimator and the kernel density estimator in CAT(0) orthant spaces. For the log-concave maximum likelihood estimator, we demonstrated the existence of a log-concave projection and established the continuity of the projection map, which implies the consistency of the log-concave maximum likelihood estimator. For the kernel density estimator, we introduced a modified kernel that mitigates boundary bias, ensuring uniform consistency of the estimator.

Our analysis provided a sufficient condition for the existence of the log-concave projection map when the probability measure is either empirical (corresponding to the maximum likelihood estimator) or absolutely continuous with respect to a Lebesgue-like measure in orthant space. Prior work \cite{Takazawa2024-ed} showed that conditions analogous to those in Euclidean space do not guarantee the existence of a log-concave maximum likelihood estimator. The key challenge arises from geometric differences: even when more than $p+1$ sample points are drawn from a given density, their convex hull may retain low-dimensional components, resulting in unbounded likelihoods. Identifying the precise necessary and sufficient conditions for the existence of log-concave projections remains an open problem.

Beyond existence, the extensive literature on log-concave density estimation in Euclidean space suggests several promising directions for future research in the CAT(0) orthant space setting. In particular, extending results on estimation under additional geometric constraints (e.g., symmetry, as explored in \cite{Xu2021-lu}), analyzing the rate of convergence of log-concave maximum likelihood estimators, and investigating local continuity properties of the log-concave projection (as studied in Euclidean settings by \cite{Barber2021-os}) are of particular interest.

For kernel density estimation, the search for geometrically adapted kernel functions remains an important challenge. Alternative density formulations that better align with the geometry of orthant spaces may yield improved asymptotic properties. For example, \cite{Nye2020-xb} introduced a Gaussian-type density incorporating a geometric random walk, exhibiting a characteristic ``bend'' at the boundaries. Similarly, densities derived from the multispecies coalescent process in phylogenomics \cite{Takazawa2024-ed} share this feature. The potential advantages of such density functions warrant further investigation.

Finally, computational scalability remains a significant challenge, particularly in high-dimensional settings, as contemporary phylogenetic analyses often involve hundreds or thousands of taxa. While nonparametric density estimation in high-dimensional spaces is generally computationally prohibitive, methods such as the log-concave maximum likelihood estimator proposed in \cite{Xu2021-lu} mitigate this issue by constraining the shape of the density contours, effectively reducing the problem to a lower-dimensional one. Adapting similar techniques to orthant spaces while determining appropriate contour shapes is an intriguing direction for future research.

\begin{appendix}

\section{Proof of Results in Section 3}
\label{sec:proof-logcon}
In this section, we give a proof of the main result in Section 3.
\subsection{Lemmas}
\begin{lemma}[Generalization of \cite{Dumbgen2011-mq} Lemma 4.3]\label{lem:proj-set}
 Assume that a metric space $(\mathcal{H}, d)$ is Heine-Borel. Let $(S_n)_n$ be a sequence of nonempty closed sets in $\mathcal{H}$ such that $d(x_0, S_n) = O(1)$ for some $x_0\in\mathcal{H}$. Then, there exists a subsequence $(S_{n(k)})_k$ and a nonempty subset $S\subseteq \mathcal{H}$ such that 
    \begin{align}
        \lim_{k\to\infty} d(y, S_{n(k)}) &= d(y,S) & \forall y\in\mathcal{H}.
    \end{align}
\end{lemma}

\begin{proof}
    For an arbitrary $x,y\in \mathcal{H}$, we have
    \begin{align}
    \begin{aligned}
        &d(x, z) - d(y,z) &&\leq d(x,y) && \forall z\in S_n \nonumber\\
       & \left(\inf_{z^\prime\in S_n} d(x,z^\prime)\right) - d(y,z)&&\leq d(x,y) && \forall z \in S_n \nonumber\\
     & d(x,S_n) - \inf_{z\in S_n} d(y,z) &&\leq d(x,y) \nonumber\\
      & d(x,S_n) - d(y,S_n) &&\leq d(x,y).
    \end{aligned}
    \end{align}
    This implies that
    \begin{align}
        |d(x,S_n) - d(y,S_n)|\leq d(x,y).
    \end{align}
    This shows uniform Lipschitz continuity. Also, it implies that for any $y\in\mathcal{H}$, $d(y,S_n) = O(1)$. Therefore, for all $y\in \mathcal{H}$, $(d(y,S_n))_n$ has a convergent subsequence (Bolzano--Weierstrass). 
    Since $(\mathcal H,d)$ is Heine--Borel, it is $\sigma$-compact and hence separable; fix a countable dense set $D\subset\mathcal H$. Then there exists a subsequence $(S_{n(k)})_k$ of $S_n$ such that for every $y\in D$, $(d(y,S_{n(k)}))_k$ converges. Further, by Lipschitz continuity, the same holds for any $y\in\mathcal{H}$.

    Now fix $y\in\mathcal H$. By the Heine--Borel property, there exists a projection $x_k\in S_{n(k)}$ of $y$; that is, $d(y,x_k)=d(y,S_{n(k)})$.
 Because $(d(y,x_k))_k$ is bounded, there exists $R>0$ such that $x_k \in B_R(y)$. By the Heine--Borel property, $(x_k)_k$ has a convergent subsequence $x_{k(l)}$ to some point $x$. Then, 
    \begin{align*}
        G(x) \coloneqq \lim_{k\to\infty} d(x,S_{n(k)}) = \lim_{l\to\infty} d(x,S_{n(k(l))}) \leq \lim_{l\to\infty}d(x,x_{k(l)}) = 0.
    \end{align*}

    The set $S = \{G = 0\}$ is closed, because for any sequence $\{z_m\}\subset S$ with $z_m\to z$,
    \begin{align*}
        \lim_{k\to\infty} d(z,S_{n(k)}) \leq d(z,z_m) + \lim_{k\to \infty} d(z_m, S_{n(k)}) = d(z,z_m) \to 0 \quad (\text{as } m\to\infty).
    \end{align*}
    Now, 
    \begin{align*}
        G(y) &= \lim_{l\to\infty} d(y,S_{n(k(l))}) = \lim_{l\to\infty} d(y,x_{k(l)}) = d(y,x) \geq d(y,S).
    \end{align*}
    On the other hand, for any $x^\prime\in S$, 
    \begin{align}
         G(y) &= \lim_{k\to\infty} d(y,S_{n(k)}) \leq \lim_{k\to\infty} d(x^\prime, S_{n(k)}) + d(y,x^\prime) = d(y,x^\prime), \nonumber\\
         \therefore \quad G(y) &\leq d(y,S).
    \end{align}
    This shows that $G(y) = d(y, S)$.
\end{proof}

\begin{proposition}[\cite{Bacak2014-nn}, Proposition 2.1.16]\label{prop:intersection}
    Let $I$ be an arbitrary directed set, and let $(C_i)_{i\in I}$ be a nonincreasing family of bounded closed nonempty convex sets in a Hadamard space $\mathcal{H}$. Then, $\cap_{i\in I}C_i \neq \emptyset$.
\end{proposition}

From this point onward we work in the CAT(0) orthant space $\mathcal{O}(\mathcal{E},\Omega)$
with its geodesic metric and reference measure $\nu$, although some results extend to more general settings. 
The next lemma shows that functions in $\Phi$ always attain a maximum.

\begin{lemma} \label{lem:sup_attain}
    For any $\phi\in\Phi$, $\sup_{x\in \mathcal{O}(\mathcal{E}, \Omega)}\phi(x)$ is attained.
\end{lemma}
\begin{proof}
$\phi$ is bounded above on bounded sets. To see this, assume $\phi$ is not bounded above on a closed ball $\bar{B}_r(x_0) = \{x\in \mathcal{O}(\mathcal{E}, \Omega)\mid d(x,x_0)\leq r\}$. Because $S_k = \{x\in \bar{B}_{r}(x_0)\mid \phi(x)\geq k\}$ is nonempty, Proposition \ref{prop:intersection} implies $\cap_{k=1}^\infty S_k \neq \emptyset $. This contradicts $\phi\in\Phi$.

    Combined with the coercivity of $\phi$, we can see that $\phi$ is bounded above globally. Thus, $\sup_{x\in \mathcal{O}(\mathcal{E}, \Omega)}\phi(x) < \infty$. Now, Proposition \ref{prop:intersection} shows that $\cap_{n=1}^\infty \{x \in\mathcal{O}(\mathcal{E}, \Omega) \mid \phi(x) \geq \sup_{y\in \mathcal{O}(\mathcal{E}, \Omega)}\phi(y) - 1/n \} \neq \emptyset$, and thus a maximizer exists.
\end{proof}

Any function in $\Phi$ has a majorizing ``linear'' function.
\begin{lemma}[Generalization of \cite{Cule2010-uy}, Lemma 1]\label{lem:majorize_func}
        For any $\phi\in\Phi$ and $x_0\in \dom \phi$, there exists $a>0, b\in\mathbb{R}$ such that $\phi(x)\leq -ad(x,x_0) + b$.
\end{lemma}
\begin{proof}
    Because $\phi$ is coercive, there exists $R>0$ such that if $d(x,x_0)\geq R$ then $\phi(x)-\phi(x_0) \leq -1$.

    Take $x$ with $d(x,x_0)>R$. We have 
    \begin{align}
    \frac{\phi(x)}{ d(x,x_0) } = \frac{ \phi(x) - \phi(x_0) }{ d(x,x_0) } + \frac{\phi(x_0)}{d(x,x_0)} .\label{eq:lem_upperbound1}
    \end{align}
    
    Now, for $\lambda\in(0,1)$, let $y=\gamma_{x_0, x}(\lambda)$. Then, by concavity,
    \begin{align}
        \phi(y) &\geq \lambda \phi(x) + (1-\lambda) \phi(x_0) \nonumber  \\
        \phi(y) - \phi(x_0) &\geq \lambda(\phi(x) - \phi(x_0)) \nonumber \\
        \frac{\phi(y) - \phi(x_0)}{\lambda d(x,x_0)} &\geq \frac{\phi(x) - \phi(x_0)}{d(x,x_0)} \nonumber \\  \therefore  \frac{\phi(y) - \phi(x_0)}{d(y,x_0)} &\geq \frac{\phi(x) - \phi(x_0)}{d(x,x_0)} .
    \end{align}

    Therefore, the function $t\mapsto \sup_{x: d(x,x_0) = t}\{\phi(x)-\phi(x_0)\}/t$ is nonincreasing. This implies 
    \begin{align}
        \frac{\phi(x)-\phi(x_0)}{d(x,x_0)} \leq \sup_{y: d(y,x_0)=R} \frac{\phi(y)-\phi(x_0)}{R}\leq -\frac{1}{R}. \label{eq:lem_upperbound2}
    \end{align}

    The equations \eqref{eq:lem_upperbound1} and \eqref{eq:lem_upperbound2} imply
    \begin{align}
        \phi(x) \leq -\frac{1}{R} d(x,x_0) + \phi(x_0) && \forall x : d(x,x_0) \geq R.
    \end{align}

    For the rest, use Lemma \ref{lem:sup_attain} and let $a=1/R, b = \sup_{x\in \mathcal{O}(\mathcal{E}, \Omega)}(\phi(x)) + 1$.
\end{proof}

\begin{lemma}[Generalization of Lemma 2.1 of \cite{Dumbgen2010-uj}]\label{lem:closedconvex_nu}
Assume that $Q$ has a density with respect to $\nu$. Then, for any positive orthant $O_F$ with $Q(O_F)> 0$,
\[
\lim_{\delta\downarrow 0}\ \sup\{\,Q(C\cap O_F)\mid C\subseteq \mathcal{O}(\mathcal{E},\Omega)\ \text{closed and convex},\ \nu(C)\le \delta\,\}
\;<\; Q(O_F).
\]
\end{lemma}

\begin{proof}
First, let $S_F = O_{F} \cap \supp(Q)$. We can take $(p+1)$ points  $x_0, \ldots, x_p \in S_F $ such that $\Delta = \conv (x_0, \ldots, x_p)$ has nonempty interior. For simplicity, we regard these points as points in Euclidean space. Take $z\in \inter(\Delta)$ and for $\alpha\in (0,1/2) $, let $\tilde{x}_i = (1-\alpha) x_i + \alpha z $ and 
\begin{align}\tilde{\Delta} = \conv( \tilde{x_0}, \ldots, \tilde{x_p} ).\end{align}

Define the $k$th corner simplex of $\tilde{\Delta}$ by:
\begin{align} \Delta_k = \{2\tilde{x}_k - x \mid x\in \tilde{\Delta}, 2\tilde{x}_k - x \in \mathbb{R}^p_{\geq 0}\}. \end{align}

Then, by the arguments in Euclidean space and using that $x_k$ is in the positive orthant, we have $x_k\in\inter (\Delta_k) $.

Now, by taking $\epsilon>0$ sufficiently small, the open ball $B_\epsilon (x_k)$ can be contained in $\inter( \Delta_k )$. 
Then, by the minimality of support, $Q(B_\epsilon (x_k)) > 0$. Let $\eta \coloneqq \min_{k=0,\ldots, p} Q( \Delta_k ) > 0$.

Any closed convex set $C$ with $Q(C\cap O_F) > Q(O_F)-\eta$ must intersect all $\Delta_k$. Thus, we can take $x_k^\ast \in \Delta_k\cap C$. Now, $\Delta^\ast =  \conv ( x_0^\ast, \ldots, x_p^\ast )$ contains $\tilde{\Delta}$. Therefore, 
\begin{align}
    \nu(C\cap O_F) \geq \nu(\Delta^\ast) \geq \nu( \tilde{\Delta} )> 0.
\end{align}

This shows that any closed convex set $C$ with $\nu(C\cap O_F) < \nu(\tilde{\Delta})$ has $Q(C\cap O_F) \leq Q(O_F)-\eta$.
\end{proof}

A similar property can be proved for a sequence of measures $Q_n$ that converges weakly to $Q$.

\begin{lemma}\label{lem:closeconvex_nu_weak}
    Assume that a probability measure $Q$ has a density with respect to $\nu$ and a sequence of probability measures $Q_n$ converges weakly to $Q$. Then, for any positive orthant $O_F$ with $Q(O_F) > 0$, there exist $N\in\mathbb{N}$ and $\eta, \delta>0$ such that
    \[\sup\{Q_n(C\cap O_F)\mid C\subseteq \mathcal{O}(\mathcal{E},\Omega): \text{ closed and convex }, \nu(C)\leq \delta, n\geq N\} \leq Q(O_F) - \eta.\]
\end{lemma}

\begin{proof}
As in the proof of Lemma \ref{lem:closedconvex_nu}, there exists a simplex with positive $\nu$-measure $\tilde{\Delta} \subseteq O_F$, a positive number $\xi>0$, and a set of open balls with $Q$-measure greater than $\xi$, $B_\epsilon(x_k) \subseteq O_F ~(k=0, \ldots, p)$, such that for any closed convex set $C$ with $C\cap B_{\epsilon}(x_k)\neq \emptyset$ for all $k$, $\tilde{\Delta}\subseteq C$. 

On the other hand, by the weak convergence, $\liminf_{n\to\infty} Q_n(B_{\epsilon}(x_k)) \geq Q(B_\epsilon(x_k))$ and $\lim_{n\to\infty} Q_n(O_F) = \lim_{n\to\infty}Q_n(\bar{O}_F) = Q(O_F)$. In particular, for any positive number $\tau < \xi/2$, there exists $N$ such that if $n\geq N$, $Q_n(B_{\epsilon}(x_k)) > \xi - \tau$ for all $k$ and $Q_n(O_F)<Q(O_F) + \tau$.

If $\nu(C)<\nu(\tilde{\Delta})\eqqcolon \delta$, since $\tilde{\Delta}\not\subseteq C$, there exists $k$ such that $C\cap B_\epsilon(x_k) = \emptyset$. Thus, if $n\geq N$, $Q_n(C\cap O_F) \leq Q_n(O_F)-\min_{k=0,\ldots,p}Q_n(B_\epsilon(x_k)) < Q(O_F) - \xi + 2\tau$. Taking $\eta=\xi-2\tau$ concludes the proof.
\end{proof}

For $\phi\in\Phi$ and $t\in\mathbb{R}$, let $D_t^\phi = \{x\in \mathcal{O}(\mathcal{E},\Omega)\mid \phi(x)\geq t\}$. 
\begin{lemma}\label{lem:level_nu}
    Assume that $\phi\in\Phi$ satisfies $\int e^\phi d\nu = 1$. Then, for $r<M\leq \max_{x\in O_F}\phi(x)$,
    \begin{align}\nu(D_r^\phi\cap O_F) \leq \frac{(M-r)^p e^{-M}}{\int_{0}^{M-r} t^p e^{-t} dt }.\end{align}
\end{lemma}
\begin{proof}
    This is essentially Lemma 4.1 of \cite{Dumbgen2010-uj}.
\end{proof}

\begin{lemma}[Generalization of Lemma 4.2 of \cite{Dumbgen2010-uj}]
   Suppose that $\bar{\phi}, \phi_1,\phi_2,\ldots\in\Phi$ satisfy $\phi_n\leq \bar{\phi}$. Also, assume that $C=\{x\in \mathcal{O}(\mathcal{E}, \Omega)\mid \liminf_{n\to\infty}\phi_n(x)>-\infty\}$ is nonempty.
   
   Then, a subsequence $(\phi_{n(k)})_k$ and a function $\phi\in\Phi$ exist with $C\subseteq\dom\phi$ and
   \begin{align}
       \lim_{k\to\infty, x\to y}\phi_{n(k)}(x) &= \phi(y) & \text{for all }y\in\inter(\dom(\phi))\cap (\cup_{F\in\Omega_p}O_F), \\
       \limsup_{k\to\infty,x\to y}\phi_{n(k)}(x) &\leq \phi(y) \leq \bar{\phi}(y) & \text{for all }y\in \mathcal{O}(\mathcal{E},\Omega).
   \end{align}\label{lem:limit_conv}
\end{lemma}
   Here, $\limsup_{k\to\infty, x\to y}f_{n(k)}(x) \coloneqq \sup \{\limsup_{k\to\infty} f_{n(k)}(y_k) \mid y_k\to y \}$. Thus, in particular, for all $y\in\inter(\dom(\phi))\cap (\cup_{F\in\Omega_p} O_F)$, 
   \begin{align}
       \lim_{k\to\infty}\phi_{n(k)}(y) &= \phi(y).
   \end{align}

\begin{proof}
    Let $G_n((x,t)) = d((x,t), \hypo(\phi_n))$. We may, without loss of generality, assume that there exists $x_0 \in \mathcal{O}(\mathcal{E},\Omega)$ and $t_0\in\mathbb{R}$ such that $(x_0, t_0)\in \hypo(\phi_n)~~(\forall n)$. Thus, $G_n((x_0,t_0)) = 0~~(\forall n)$. Then, Lemma \ref{lem:proj-set} indicates that there exists a subsequence $(G_{n(k)})_k$ and a closed set $S\subseteq \mathcal{O}(\mathcal{E},\Omega)\times\mathbb{R}$ such that for any $(x,t)\in \mathcal{O}(\mathcal{E},\Omega)\times\mathbb{R}$, $\lim_{k\to\infty} G_{n(k)}((x,t)) = d((x,t), S)\eqqcolon G((x,t))$. Here, $S = \{ (x,t)\in \mathcal{O}(\mathcal{E},\Omega)\times\mathbb{R}\mid \lim_{k\to\infty} G_{n(k)}((x,t)) = 0 \}$.

    Now, the set $S$ is convex and satisfies $(x,t)\in S, s<t \Rightarrow (x,s)\in S$. Note that the convexity of $S$ follows from the CAT(0) property. Thus, $S$ is a hypograph of an upper-semicontinuous concave function. We define $\phi$ to be this function.

    We first see that for any $y\in \mathcal{O}(\mathcal{E},\Omega)$, $\limsup_{k\to\infty, x\to y}\phi_{n(k)}(x) \leq \phi(y)$. Let $\limsup_{k\to\infty, x\to y}\phi_{n(k)}(x) = t$. Observe that there exists a further subsequence $\phi_{n(k(l))}$ with $\sup\{ \phi_{n(k(l))}(x) \mid d(x,y)< 1/l \} \to_{l\to\infty} t$. Then, 
    \begin{align}&\limsup_{l\to\infty} d((y,t), \hypo(\phi_{n(k(l))})) \nonumber\\
    \leq &\limsup_{l\to\infty} |t - \sup\{ \phi_{n(k(l))}(x) \mid d(x,y)< 1/l \}| + \frac{1}{l} = 0.
    \end{align}
    This implies $\lim_{k\to\infty} G_{n(k)}((y,t)) = 0$ as well, leading to $(y,t)\in S$. This shows $\phi(y)\geq t$. In particular, $C\subseteq \dom(\phi)$ holds.
    
    Now, since $d((x,t), \hypo(\phi_{n(k)})) \geq d((x,t),\hypo(\bar{\phi}))$ holds for any $x, t$ and $k$, $d((x,t),\hypo(\phi)) \geq d((x,t),\hypo(\bar{\phi}))$ follows. This 
    means that $\phi\leq \bar{\phi}$.

    For simplicity, we retake $(\phi_{n(k)})_k$ as $(\phi_n)_n$. 


    There exist $F\in\Omega_p$ and $y\in\inter(\dom(\phi))\cap O_F$ such that for sufficiently small $\delta>0$ and $b_0,\ldots,b_p \in \mathbb{R}^p$, $\conv(y + \delta b_0, \ldots, y + \delta b_p)\subseteq \dom(\phi)\cap O_F$, $y\in\inter(\conv(y + \delta b_0, \ldots, y + \delta b_p))$. Now, $\lim_{n\to\infty}G_n((y + \delta b_j, \phi(y + \delta b_j))) = G((y + \delta b_j, \phi(y + \delta b_j))) = 0$ indicates that for any $\epsilon>0$, there exists some $N\in\mathbb{N}$ such that
    \[n\geq N \Rightarrow \forall j, \exists x_{n,j} \in \mathcal{O}(\mathcal{E},\Omega), d(y + \delta b_j, x_{n,j}) < \epsilon \text{ and } \phi_n(x_{n,j}) > \phi(y + \delta b_j) - \epsilon.\]
    This means that there exists some sequence $\{x_{n,j}\}_n$ such that $\liminf_{n\to\infty} \phi_n(x_{n,j}) \geq \phi(y + \delta b_j)$, $x_{n,j}\to y + \delta b_j$.

    Now, if we let $\Delta_n = \conv(x_{n,0},\ldots, x_{n,p})$, it contains a ball around $y$ eventually. Therefore, 
    \begin{align}
        \liminf_{n\to\infty,x\to y}\phi_n(x) &= \inf \{\liminf_{n\to\infty} \phi_n(y_n) \mid y_n \to y\}  \nonumber \\
        &\geq \liminf_{n\to\infty}\left[ \inf_{x\in\Delta_n} \phi_n(x)\right] \nonumber \\
        &= \liminf_{n\to\infty} \min_{x\in\Delta_n}\phi_n(x) \nonumber \\
        &= \liminf_{n\to\infty} \min_{j=0,\ldots,p} \phi_n(x_{n,j}) \nonumber \\
        &\geq \min_{j=0,\ldots,p}\phi(y + \delta b_j).
    \end{align}
    By letting $\delta\to 0$, the right-hand side converges to $\phi(y)$ by continuity.
\end{proof}

\begin{lemma}[Generalization of Lemma 2.13 of \cite{Dumbgen2010-uj}]\label{lem:Q_n_h}
    Assume that $Q$ has a density, and $(Q_n)_n$ converges weakly to $Q$. Also, let $\inter(\supp(Q))\cap \cup_{F\in\Omega_p} O_F\neq \emptyset$. Let the functional $h$ be as defined in \eqref{eq:def_of_h}.
    
    Then, for any $x\in\inter(\csupp(Q))\cap \cup_{F\in\Omega_p} O_F$, \[\limsup_{n\to\infty} h(Q_n,x) \leq h(Q,x) < 1.\]
\end{lemma}

\begin{proof}
Let $\mathcal{G}(x) = \{C : \text{closed and convex }, x\notin \inter(C)\}$. To derive a contradiction, assume \(\limsup_{n\to\infty} h(Q_n,x) \ge h(Q,x)+\eta\) for some \(\eta>0\).
Passing to a subsequence (not relabeled), we may assume \(h(Q_n,x)\to h(Q,x)+\eta\). Then, there exists a sequence of closed convex sets $(S_n)_n$ such that $S_n\in \mathcal{G}(x)$ and $\lim Q_n(S_n) = h(Q,x)+\eta$. 

Since $Q$ is tight, we can take a point $y\in \mathcal{O}(\mathcal{E},\Omega)$ and an open ball $B_{R}(y)$ such that $Q(B_{R}(y)) > 1- h(Q,x)$. (Note that $0<h(Q,x)$ always holds for a measure $Q$ that has a density.) Now, by weak convergence, $\liminf Q_n(B_R(y))\geq Q(B_R(y)) > 1-h(Q,x)$. This means that for sufficiently large $n$, $S_n\cap B_R(y)\neq \emptyset$. Thus, $d(y,S_n) = O(1)$.
Applying Lemma 
 \ref{lem:proj-set}, there exists a subsequence $(S_{n(k)})_k$ and a nonempty closed set $S$ such that for any $z$, $\lim_{k\to\infty} d(z,S_{n(k)}) = d(z,S)$. In particular, this $S$ is closed and convex. 

Now, define the $\delta$-enlargement of $S$ by $S^{(\delta)} = \{z\in \mathcal{O}(\mathcal{E},\Omega) \mid d(z,S)\leq \delta\}$. Then, this is closed and convex. Also, by the weak convergence, 
    \begin{align}
        Q(S^{(\delta)}) &\geq \limsup_{k\to\infty} Q_{n(k)}(S^{(\delta)})\nonumber \\
        &\geq \limsup_{k\to\infty} Q_{n(k)}(S^{(\delta)} \cap S_{n(k)}) \nonumber \\
        &= \lim Q_{n(k)}(S_{n(k)}) - \liminf Q_{n(k)}(S_{n(k)} - S^{(\delta)}) \nonumber \\
        &\geq h(Q,x)+\eta - \limsup_{k\to\infty}  Q_{n(k)}\left(\{ z\in \mathcal{O}(\mathcal{E},\Omega)\mid d(z, S_{n(k)}) = 0 \text{ and } d(z,S) > \delta \}\right). \label{eq:Q-S_delta}
    \end{align}

    Because $d(\cdot, T)$ is Lipschitz continuous for any closed set $T$, $d(z,S_{n(k)})$ converges uniformly to $d(z, S)$ on compact sets. In addition, because $(Q_{n})_n$ converges weakly to a tight measure $Q$, they are tight. Therefore, for any $\epsilon>0$, we can take a compact set $K$ such that for any $n$, $Q_n(K) > 1-\epsilon$. Now, by the uniform convergence on $K$, there exists $L\in\mathbb{N}$ such that if $k\geq L$ and $z\in K$, $d(z,S) < d(z,S_{n(k)}) + \delta/2$. Thus, \[\limsup_{k\to\infty}  Q_{n(k)}\left(\{ z\in \mathcal{O}(\mathcal{E},\Omega)\mid d(z, S_{n(k)}) = 0 \text{ and } d(z,S) > \delta \}\right) < \epsilon.\] 
    Because this holds for any $\epsilon>0$, the right-hand side of the equation \eqref{eq:Q-S_delta} equals $h(Q,x)+\eta$. However, the left-hand side converges to $Q(S)$ as $\delta\to 0$, leading to $Q(S)=h(Q,x)+\eta$. Note that $S$ cannot contain $x$ in its interior, because if it does, there exists a closed ball $\bar{B}_\epsilon(x)$ that is contained in $S\cap \cup_{F\in\Omega_p} O_F$. But the uniform convergence implies that for sufficiently large $k$, $\sup_{z\in\bar{B}_\epsilon(x)} d(z,S_{n(k)}) < \epsilon/2$. This contradicts the construction of $S_n$. Therefore, $S$ is a closed convex set without $x$ in its interior and with $Q$-measure $h(Q,x)+\eta$. This is a contradiction.
\end{proof}

\begin{lemma}\label{lem:lipschitz-majorization}
    For $\psi\in\Phi$, define
    \begin{align}\psi^{(M)}(x) = \sup_{y\in\mathcal{O}(\mathcal{E},\Omega)} \{\psi(y) - Md(x,y)\}.\end{align}
    Then, the following hold:
    \begin{itemize}
        \item $\psi^{(M)}$ is $M$-Lipschitz
        \item $\psi^{(M)}\in\Phi$
         \item $\psi^{(M)}(x)\searrow \psi(x)$ pointwise as $M\to\infty$.
    \end{itemize}
\end{lemma}

\begin{proof}
       To see the first claim, take any two points $x,y$. Then, at $x$, for any $\epsilon>0$, there exists $z$ with
    \begin{align}
        \psi^{(M)}(x) < \psi(z) - Md(x,z) + \epsilon.
    \end{align}
    Thus, 
    \begin{align}
        \psi^{(M)}(y) &\geq \psi(z) - Md(y,z) \nonumber \\
        &\geq \psi(z) - M(d(x,y) + d(x,z)) \nonumber \\
        & > \psi^{(M)}(x) - \epsilon - Md(x,y). 
    \end{align}
     By letting $\epsilon \searrow 0$, $\psi^{(M)}(y) \geq \psi^{(M)}(x) - Md(x,y)$. The converse can be proven similarly.

     For the second claim, we first verify its concavity. For any $x_0, x_1, y_0, y_1 \in \mathcal{O}(\mathcal{E},\Omega)$, let $x_t = (1-t) x_0 + tx_1, y_t = (1-t) y_0 + ty_1$. Then,
     \begin{align}
        \psi^{(M)}(x_t) &\geq \psi(y_t) - Md(x_t, y_t) \nonumber \\
        &\geq (1-t) \psi(y_0) + t\psi(y_1) - M(1-t)d(x_0,y_0) - Mtd(x_1, y_1) &  \nonumber \\
        &= (1-t) \{ \psi(y_0) - Md(x_0,y_0) \} + t\{ \psi(y_1) - Md(x_1,y_1) \}.
    \end{align}

    Because $y_0,y_1$ are arbitrary, by taking the supremum over them, 
     \begin{align}
        \psi^{(M)}(x_t) \geq (1-t) \psi^{(M)}(x_0) + t \psi^{(M)}(x_1).
    \end{align}

    Now we verify its coercivity. Because $\psi$ is coercive, for any $x_0$ and $b\in\mathbb{R}$, there exists $R>0$ such that if $d(x,x_0)\geq R$, $\psi(x) < b$. Also, $\sup \psi(x) = \sup \psi^{(M)}(x)$ and both are bounded. Let the bound be $L$. If we let $T = R + (L+|b|)/M$, for $x$, $y$ with $d(x,x_0) > T$ and $d(x_0,y)< R$, the triangle inequality implies $d(x,y)>(L+|b|)/M$, and
     \begin{align}
         \psi(y) - Md(x,y) < L - M \left(\frac{L+|b|}{M}\right) = -|b| \leq b.
     \end{align}
Therefore, if $d(x_0,x)>T$, 
      \begin{align}
         \psi^{(M)}(x) \leq b.
     \end{align}
     This shows the coercivity.

     For the last claim, take an arbitrary $x\in \mathcal{O}(\mathcal{E},\Omega)$. By the upper semicontinuity, for any $\epsilon>0$, there exists $\delta >0$ such that if $d(x,y)\leq \delta$, $\psi(x) > \psi(y) - \epsilon$. 
     Here, because $\psi$ is bounded above by $L$, if $M > (L - \psi(x))/\delta> 0$, for $y$ with $d(x,y)\geq \delta$, 
     \begin{align}
         \psi(y) - Md(x,y) &\leq \psi(y) - M\delta \nonumber \\
         &< \psi(y) - (L - \psi(x)) \nonumber \\
         &= \psi(x) - (L-\psi(y)) \leq \psi(x).
     \end{align}
     Thus, 
     \begin{align}
         \psi^{(M)}(x) &= \sup_{y: d(x,y)<\delta}\{ \psi(y) - Md(x,y) \} \nonumber \\
         &\leq \sup_{y: d(x,y)<\delta}\psi(y) < \psi(x) + \epsilon.
     \end{align}

     We have now shown all three claims.
\end{proof}

\subsection{Proof of main results}
\begin{proof}[Proof of Lemma \ref{lem:Q_h}]
Take $x\in\inter({\csupp(Q)})\cap \cup_{F\in\Omega_p}O_F$ and let
\begin{align}
    \mathcal{G}(x) = \{C : \text{closed and convex }, x\notin \inter(C)\}.
\end{align}
For arbitrary $C\in\mathcal{G}(x)$, $Q(C)<1$ is obvious from the minimality of convex support.

Assume $h(Q,x)=1$, i.e., there exists a sequence of closed convex sets $(S_N)_N\subseteq \mathcal{G}(x)$ and
$Q(S_N) \nearrow 1$.
Here, if we let $T_N = S_N\cap \csupp(Q)$, $(T_N)_N\subseteq \mathcal{G}(x)$ and $Q(T_N)\nearrow 1$. Now, we define functions $\psi_N$ and $\bar{\psi}$ as follows:
\begin{alignat}{2}
    \psi_N(x) &= \log \mathbf{1}\{x\in T_N\} &&\coloneqq \begin{cases} 0 & \text{if }x\in T_N \\ -\infty & \text{otherwise} \end{cases}, \\
    \bar{\psi}(x) &= \log \mathbf{1}\{x\in \csupp(Q)\} &&\coloneqq \begin{cases} 0 & \text{if }x\in \csupp(Q) \\ -\infty & \text{otherwise}. \end{cases}
\end{alignat}
By definition, $\psi_N \leq \bar{\psi}$. 

Because $Q$ has a density, there exists a maximal orthant $O_F$ in which $\supp(Q)$ is not contained in any hyperplane. Thus, we can take $(p+1)$ points $y_0, \ldots, y_p \in \supp(Q)\cap O_F$ such that its convex hull has nonempty interior. Let one of the interior points be $y$. Then, any subspace of $O_F$ given as  $H(u,y) = \{x\in O_F \mid u^\top y \leq u^\top x \}$ for some $u\in\mathbb{R}^p\backslash\{0\}$ separates points $y_0,\ldots,y_p$ into two parts. This shows that $Q(H(u,y))$ is positive for each unit vector $u$. By continuity of the integral, $\inf_{\|u\|=1} Q(H(u,y))$ is attained, and hence positive, say equals $c>0$. Then, any closed convex set that has a probability greater than $1-c$ must contain $y$ in its interior. This implies that $y$ is eventually in $T_N$, thus $\liminf_{N\to\infty} \psi_N(y) = 0$.

Now we can use Lemma \ref{lem:limit_conv}. There exists a subsequence $(\psi_{N(k)})_k$ of $(\psi_N)_N$ and a function $\psi\in\Phi$ such that
\begin{align}
    \lim_{k\to\infty, z\to y} \psi_{N(k)}(z) &= \psi(y) & \forall y \in \inter(\dom(\psi))\cap \cup_{F\in\Omega_p}
O_F,  \\   \limsup_{k\to\infty, z\to y} \psi_{N(k)}(z) &\leq \psi(y) \leq \bar{\psi}(y) & \forall y.
\end{align}
Now, by the form of the function $\psi_N$, for any $z\in\mathcal{O}(\mathcal{E},\Omega)$, $\limsup_{y\to z}\psi_N(y) = \psi_N(z)$. Note that by the convexity of $T_N$, if $T_N$ is not a singleton set, it has some element of itself in an arbitrary neighborhood of any element of it. Further, singleton sets can be excluded because $Q$ has a density. Because $T_N$ is closed,
\begin{align}
    0 &= 1 - \lim_{k\to\infty}\int \exp(\psi_{N(k)}) dQ & (\text{ Definition of }\psi_N) \nonumber\\
    &\geq \int 1- \liminf_{k\to \infty}\exp(\psi_{N(k)})dQ & (\text{Fatou's lemma}) \nonumber \\
    &= 1 - \int \limsup_{k\to\infty} \limsup_{y\to z}\exp(\psi_{N(k)}(y)) dQ(z) \nonumber \\
    &\geq 1 - \int \exp(\psi) dQ  \nonumber \\
    &\geq 1 - \int\exp(\bar{\psi})dQ= 0.
\end{align}

Here, $x\not\in \inter(\dom\psi)\cap (\cup_{F\in\Omega_p} O_F)$. This is because, if we assume $x\in\inter(\dom\psi)\cap (\cup_{F\in\Omega_p} O_F)$, $\lim_{k\to\infty, y\to x}\psi_{N(k)}(y) = 0$. However, by the construction of $\psi_N$, this implies the existence of some $K_0\in\mathbb{N}$ and $\epsilon_k>0$ such that if $k\geq K_0$, $B_{\epsilon_k}(x) \subseteq T_{N(k)}$. This contradicts the assumption that $x\not\in\inter(T_{N(k)})$.

Combining this with the observations that $\psi\leq\bar{\psi}$,  we can see that $\int\exp(\psi) dQ=1$ and $\psi\in\Phi$. This implies that $\inter(\dom\psi) = \inter(\csupp(Q))$. However, it contradicts $x\not\in\inter(\dom\psi)\cap (\cup_{F\in\Omega_p} O_F)$. 
\end{proof}
\begin{proof}[Proof of Theorem \ref{th:existence}]
    We first show that $L(Q)\in\mathbb{R}$.

    First, let $h(x) = -d(x,x_0)$. Then, 
     \[L(h,Q) = -\int d(x,x_0)dQ - \int e^{-d(x,x_0)} d\nu(x) + 1 > -\infty.\]
    This implies $L(Q)>-\infty$.

    We now restrict attention to the class of concave functions $\Phi(Q)$ defined as follows:
    \begin{align}
        \Phi(Q) = \{\phi\in\Phi \mid \int e^\phi d\nu = 1, \dom\phi\subseteq\csupp(Q)\}.
    \end{align}
    Note that these conditions are not restrictive, since we are only interested in the supremum of $L(\phi, Q)$.

    Let $\Theta = \{F\in\Omega_p\mid Q(O_F)>0\}$. Note that $\Theta$ is nonempty because $Q$ has a density. For $F\in\Theta$, let $M_F=\max_{x\in O_F}\phi(x)$. By Lemma \ref{lem:level_nu}, when $M_F>0$, we have 
    \begin{align}&\nu(D_{-cM_F}^\phi\cap O_F) \leq \frac{(1+c)^pM_F^p e^{-M_F}}{\int_0^{(1+c)M_F} t^p e^{-t}dt}, \label{eq:levelset_nu} \end{align}
    where $D_{t}^\phi = \{x\in\mathcal{O}(\mathcal{E},\Omega)\mid \phi(x)\geq t\}$ as before.
    As $M_F \to \infty$, since the denominator approaches $p!$, $\nu(D_{-cM_F}^\phi\cap O_F)$ approaches $0$.
    Now, 
    \begin{align}
        L(\phi,Q) &= \sum_{F\in\Theta} \int_{O_F} \phi dQ \nonumber \\
        &\leq \sum_{F\in\Theta} \{-cM_FQ(O_F-(D_{-cM_F}^\phi\cap O_F)) + M_FQ(D_{-cM_F}^\phi\cap O_F)\} \nonumber \\
        &=  \sum_{F\in\Theta} - cM_FQ(O_F) + cM_FQ(D_{-cM_F}^\phi\cap O_F) + M_FQ(D_{-cM_F}^\phi\cap O_F) \nonumber \\
        &= \sum_{F\in\Theta} -(c+1)M_F\left(\frac{cQ(O_F)}{c+1}-Q(D_{-cM_F}^\phi\cap O_F)\right) \nonumber \\
        &=  \sum_{F\in\Theta} -(c+1)M_F\left(Q(O_F) - \frac{Q(O_F)}{c+1}-Q(D_{-cM_F}^\phi\cap O_F)\right).
    \end{align}

    By Lemma \ref{lem:closedconvex_nu}, for sufficiently large $c$ and small $\delta$, 
    \begin{align}&\sup\{ Q(C\cap O_F) \mid C\subseteq \mathcal{O}(\mathcal{E},\Omega) : \text{ closed and convex }, \nu(C)\leq \delta \} \nonumber \\< ~&Q(O_F) - \frac{Q(O_F)}{c+1}. \end{align}
    Thus, if we let $M_F\to\infty$ for some $F$, by the equation \eqref{eq:levelset_nu}, $L(\phi,Q)\to-\infty$. Note that this holds regardless of the values of $M_G$ for any $G\neq F$. Thus, for sufficiently large $M$, we need only consider $\phi$ with $\max_{F\in\Theta}\max_{x\in O_F}\phi(x)\leq M$, but obviously $L(\phi,Q)=\int \phi dQ\leq M$. Therefore, $L(Q)<\infty$ also holds.

    Next, we show the existence of the maximizer. Take $(\phi_N)_N \subseteq \Phi(Q)$ such that $L(\phi_N,Q)\nearrow L(Q)$. Then, we can take the sequence properly such that there exists a constant $M_\ast$ such that $M_N\leq M_\ast$, where $M_N = \max_x \phi_N(x)$. 
    
    Take arbitrary $x_0\in\inter(\csupp(Q))\cap \cup_{F\in\Theta} O_F$. If $\phi_N(x_0)<M_N$, the upper level set $D_{\phi_N(x_0)}^{\phi_N}$ is a closed convex set and does not include $x_0$ as an interior point. Thus, 
    \begin{align}
           & L(\phi_N, Q) = \int \phi_N dQ \leq \phi_N(x_0) + (M_N - \phi_N(x_0))h(Q, x_0). 
    \end{align}

    This equation holds if $\phi_N(x_0)=M_N$ as well. Therefore, for any $x_0\in\inter(\csupp(Q)) \cap \cup_{F\in\Theta} O_F$, 
    \begin{align}
        \phi_N(x_0)\geq \frac{L(\phi_N, Q)-M_N h(Q,x_0)}{1-h(Q, x_0)} \geq \frac{L(\phi_1, Q)-|M_\ast|}{1-h(Q, x_0)}.
    \end{align}
    By Lemma \ref{lem:Q_h}, this is bounded below. Thus, we have just shown that
    \begin{align}
        \liminf_{N\to\infty} \phi_N(x_0) &> -\infty & \text{for any }x_0\in\inter(\csupp(Q)) \cap \cup_{F\in\Omega_p} O_F.
    \end{align}

    Let $\Theta_{\csupp}$ denote the set of positive orthants with $\csupp(Q)\cap O_F \neq \emptyset$.
    For $O_F\in\Theta_{\csupp}$ and any $x_F\in \dom \phi_N\cap O_F$, there exists $R_F>0$ such that if $x\in \bar{O}_F, d(x,x_F)\geq R_F$, $\phi_N(x) - \phi_N(x_F)\leq -1$. This can be seen as follows. First, we can take $x_F^{(0)}, \ldots, x_F^{(p)}\in \inter(\csupp(Q))\cap O_F$ such that these points are not contained in any hyperplane. Then, there exists some $m\in\mathbb{R}$ and $N_0\in\mathbb{N}$ such that if $N\geq N_0$, $\phi_N(x_F^{(i)}) \geq m~~~(\forall i)$. Now, if $R_F>0$ is sufficiently large, in order for any $\phi\in\Phi(Q)$ with $\phi(x_F^{(i)}) \geq m~~~(\forall i)$ to satisfy $\int_{O_F} \exp(\phi)d\nu \leq 1$, for any point  $x\in O_F$ with $d(x,x_F^{(0)})\geq R_F$, $\phi(x) - \phi(x_F^{(0)}) \leq -1$ must hold. Thus, in particular, if we take $x_0\in\dom\phi$, there exists a constant $R>0$ such that for any points $x\in \cup_{F\in\Theta_{\csupp}}\bar{O}_F$, if $d(x,x_0)>R$, $\phi_N(x)-\phi_N(x_0)\leq -1$ for all $N\geq N_0$. 

    The proof of Lemma \ref{lem:majorize_func} suggests that there exists $\bar{\phi}(x) = -ad(x,x_0) + b$ such that for any $N\in\mathbb{N}$, $O_F\in \Theta$ and $x\in \bar{O}_F$, $\phi_N(x)\leq \bar{\phi}(x)$. But by taking an appropriate version of $\phi_N$, $\phi_N$ can be bounded by $\bar{\phi}$ globally as well (for example by taking $\phi_N$ to be $\min(\phi_N, \bar{\phi})$).

    Now, we can apply Lemma \ref{lem:limit_conv} to $C\supseteq \inter(\csupp(Q)\cap \cup_{F\in\Omega_p} O_F)$ and $\bar{\phi}(x)=(-ad(x,x_0)+b) + \log \mathbf{1}\{x\in \csupp(Q)\}$: we can take $\psi\in\Phi$ such that 
    \begin{align}
        \limsup_{k\to\infty} \phi_{N(k)}(x) &\leq \psi(x)\leq -ad(x,x_0) + b & \text{for all }x \in \mathcal{O}(\mathcal{E}, \Omega), \\
         \lim_{k\to\infty} \phi_{N(k)}(x) &= \psi(x) > -\infty & \text{for all } x \in \inter(\csupp(Q))\cap \cup_{F\in\Omega_p} O_F.
    \end{align}
    Now, since the boundary of closed convex sets in $\mathcal{O}(\mathcal{E}, \Omega)$ also has $\nu$-measure zero, by dominated convergence, $\int e^\psi d\nu=1$. Further, if we apply Fatou's lemma to $x\mapsto -ad(x,x_0) + b - \phi_{N(k)}(x)$, 
    \begin{align}
        \limsup_{k\to\infty} \int \phi_{N(k)}dQ \leq \int \psi dQ. 
    \end{align}
    Therefore, 
    \begin{align}
        L(Q)\geq L(\psi,Q)\geq \limsup_{k\to\infty}L(\phi_{N(k)},Q) = L(Q).
    \end{align}
    Hence $L(\psi,Q) = L(Q)$ and $\psi$ is the desired maximizer.
\end{proof}

\begin{proof}[Proof of Theorem \ref{th:continuity}]
Assume first that $\lambda \coloneqq \lim_{N\to\infty} L(Q_N)$ exists and we show below that $\lambda = L(Q)$. This assumption is not restrictive because, in general, we can apply this result to a suitable subsequence to get $\liminf L(Q_N) = \limsup L(Q_N) = L(Q)$, thus $\lim_{N\to\infty} L(Q_N)$ exists anyway.

    Let $\phi(x) = -d(x,x_0)$. Then, by the convergence in Wasserstein-1 distance,
    \begin{align}
        \lambda \geq \lim_{N\to\infty} L(\phi, Q_{N}) &= -\int d(x,x_0)dQ - \int \exp(-d(x,x_0))d\nu + 1 > -\infty.
    \end{align}
    Thus, $\lambda>-\infty$, and in particular we can assume without loss of generality that $L(Q_N) > -\infty, \forall N$.

    Also, for simplicity, we assume that $L(\phi,Q_N)$ has the maximizer for all $N$, and let $\psi_N\in\Phi(Q_N)$ denote them.

    Let $\Theta=\{F\in\Omega_p\mid Q(O_F)>0\}$ and $M_N = \max_{x\in \mathcal{O}(\mathcal{E}, \Omega) \cap \cup_{F\in\Theta} O_F} \psi_N(x) \in\mathbb{R}$. We first show that the sequence $(M_N)_N$ is bounded. This would imply $L(Q)<\infty$, since $L(Q_N)  = \int \psi_N dQ_N \leq M_N$ implies $\lambda \leq \liminf_{N\to\infty} M_N$.

    Take any closed convex set $C$. If $\nu(C)$ is sufficiently small, for any positive orthant $O_F$ for $F\in\Theta$, there exists $N_0\in\mathbb{N}$ and $\eta>0$ such that if $N\geq N_0$, $Q_N(C\cap O_F) \leq Q(O_F) - \eta$ by Lemma \ref{lem:closeconvex_nu_weak}.

    In particular, if we take a sequence of closed convex sets $\{C_N\}$ such that $\nu(C_N)\to 0$, for any positive orthant $O_F$ with $Q(O_F)>0$, $\limsup_{N\to\infty}Q_N(C_N\cap O_F) < Q(O_F)$.

    For $F\in\Theta$, let $M_F^{(N)} = \max_{x\in O_F} \psi_N(x)$. By Lemma \ref{lem:level_nu}, \begin{align}\nu(D_{-cM_F^{(N)}}^{\psi_N}\cap O_F) &\leq \frac{(1+c)^p (M_F^{(N)})^p e^{-M_F^{(N)}}}{p! + o(1)} &(\text{ as } M_F^{(N)}\to\infty). \label{eq:bound_levelset}
    \end{align}
    On the other hand,
        \begin{align}
            &L(\psi_N, Q_N) \nonumber \\
            =& \sum_{F\in\Theta}\int_{O_F}\psi_N dQ_N \nonumber\\
            \leq& \sum_{F\in\Theta}\{-cM_F^{(N)} Q_N(O_F - (D_{-cM_F^{(N)}}^{\psi_N}\cap O_F)) + M_F^{(N)} Q_N(D_{-cM_F^{(N)}}^{\psi_N} \cap O_F)\} \nonumber\\
    =& \sum_{F\in\Theta}-cM_F^{(N)} Q_N(O_F) + (c+1)M_F^{(N)} Q_N(D_{-cM_F^{(N)}}^{\psi_N}\cap O_F)\nonumber \\
        =& \sum_{F\in\Theta} -(c+1)M_F^{(N)}\left( \frac{cQ_N(O_F)}{c+1} - Q_N(D_{-cM_F^{(N)}}^{\psi_N}\cap O_F) \right)\nonumber\\
        =& \sum_{F\in\Theta} (c+1)M_F^{(N)}\left( - Q_N(O_F) + \frac{Q_N(O_F)}{c+1} + Q_N(D_{-cM_F^{(N)}}^{\psi_N}\cap O_F) \right). \label{eq:l_bound}
       \end{align}

    Now, take $c$ sufficiently large such that $1/(c+1) < \eta/2$. Then by equation \eqref{eq:l_bound}, 
       \begin{align}
            L(\psi_N, Q_N) &\leq\sum_{F\in\Theta}(c+1)M_F^{(N)}\left( - Q_N(O_F) + \eta/2 + Q_N(D_{-cM_F^{(N)}}^{\psi_N}\cap O_F) \right). \label{eq:c_adjusted}
       \end{align}
       But the equation \eqref{eq:bound_levelset} and the previous argument shows that there exists $M\in\mathbb{R}$ such that if $N\geq N_0$ and $M_F^{(N)}\geq M$, $Q_N(D_{-c M_F^{(N)}}^{\psi_N}\cap O_F) < Q(O_F)-\eta.$ Thus, each summand of the equation \eqref{eq:c_adjusted} converges to $-\infty$ uniformly over $N$ as $M_F^{(N)}\to \infty$. Therefore, if $N\geq N_0$, (as long as we let $\psi_N$ vary inside $\Phi(Q)$), each summand is bounded above. Thus, in particular, $(M_N)_N$ has to be bounded above.\footnote{
      Suppose $(M_N)_N$ is not bounded above. Then, there exists a subsequence $(M_{N(k)})_k$ and a positive orthant $F\in\Theta$ such that $\lim_{k\to\infty}M_{N(k)} = \lim_{k\to\infty}\max_{x\in O_F}\psi_N(x) = \infty$. Then, the summand with regard to $F$ approaches $-\infty$, while other terms are bounded above. Thus, $\liminf L(\psi_N, Q_N) = \lim L(\psi_N, Q_N) =  -\infty$.
       } Note also that we can take a version of $\psi_N$ to satisfy $\psi_N\leq M_N$ globally.

    Next, we show that the sequence of functions $\psi_N$ can be bounded above uniformly by an element of $\Phi$. 
    Take arbitrary $x_0\in\inter(\csupp(Q))\cap\cup_\Theta O_F$. If $\psi_N(x_0)<M_N$, the upper level set $D_{\psi_N(x_0)}^{\psi_N}$ is a closed convex set and does not include $x_0$ as an interior point. Thus, 
    \begin{align}
           & L(Q_N) = \int \psi_N dQ_N \leq \psi_N(x_0) + (M_N - \psi_N(x_0))h(Q_N, x_0). 
    \end{align}
    This equation holds for $x_0$ with $\psi_N(x_0)=M_N$ as well. Therefore, for any $x_0\in\inter(\csupp(Q))\cap\cup_\Theta O_F$, 
    \begin{align}
        \psi_N(x_0)\geq \frac{L(Q_N)-M_N h(Q_N,x_0)}{1-h(Q_N, x_0)} \geq \frac{L(Q_N)-|M_N|}{1-h(Q_N, x_0)}.
    \end{align}
    By Lemma \ref{lem:Q_n_h} and the boundedness of $M_N$, for sufficiently large $N$, this is bounded below. Thus, we have just shown that
    \begin{align}
        \liminf_{N\to\infty} \psi_N(x_0) &> -\infty & \text{for any }x_0\in\inter(\csupp(Q))\cap \cup O_F.
    \end{align}

    Let $\Theta_{\csupp}$ denote the set of positive orthants with $\csupp(Q)\cap O_F \neq \emptyset$. Fix $x_0\in \inter(\csupp(Q))\cap \cup O_F$.
    Then, for $O_F\in\Theta_{\csupp}$, there exists $R_F>0$ and $x_F\in O_F$ such that if $x\in \bar{O}_F$ and $d(x,x_F)\geq R_F$, $\psi_N(x) - \psi_N(x_0)\leq -1$. This can be seen as follows. First, we can take $x_F^{(0)}, \ldots, x_F^{(p)}\in \inter(\csupp(Q))\cap O_F$ such that these points are not contained in any hyperplane. Then, there exists some $m\in\mathbb{R}$ and $N_F\in\mathbb{N}$ such that if $N\geq N_F$, $\psi_N(x_F^{(i)}) \geq m$ for all $i=0,\ldots,p$ and $\psi_N(x_0)\geq m$. Now, if $R_F>0$ is sufficiently large, in order for any $\phi\in\Phi(Q)$ with $\phi(x_F^{(i)}) \geq m~(\forall i), \phi(x_0)\geq m$ to satisfy $\int_{O_F} \exp(\phi)d\nu \leq 1$, for any point  $x\in \bar{O}_F$ with $d(x,x_F^{(0)})\geq R_F$, $\phi(x)-\phi(x_0) \leq \phi(x) - m \leq -1$ needs to hold. 
    
    Thus, in particular, there exists a constant $R>0$ and $N_0\in\mathbb{N}$ such that for any points $x\in \cup_{\Theta_{\csupp}}\bar{O}_F$, if $d(x,x_0)>R$, $\psi_N(x)-\psi_N(x_0)\leq -1$ for all $N\geq N_0$.

    Now, the proof of Lemma \ref{lem:majorize_func} implies that there exists $\bar{\psi}(x) = -ad(x,x_0) + b$ such that for any $N\geq N_0$, $O_F\in \Theta_{\csupp}$ and $x\in \bar{O}_F$, $\psi_N(x)\leq \bar{\psi}(x)$. But by taking an appropriate version of $\psi_N$, $\psi_N$ can be bounded by $\bar{\psi}$ globally as well (for example by retaking $\psi_N$ to be $\min(\psi_N, \bar{\psi})$).

    Now we apply Lemma \ref{lem:limit_conv} to $C\supseteq \inter(\csupp(Q)\cap \cup_{F\in\Omega_p} O_F)$ and $\bar{\phi}(x)=(-ad(x,x_{0})+b)$, and we can get $\psi\in\Phi$ with
    \begin{align}
        \limsup_{k\to\infty,x\to y} \psi_{N(k)}(x) &\leq \psi(y)\leq -ad(y,x_{0}) + b ~~~~~~~ \text{for all }y \in \mathcal{O}(\mathcal{E}, \Omega), \\
         \lim_{k\to\infty,x\to y} \psi_{N(k)}(x) &= \psi(y) > -\infty ~~~~~~~ \text{for all } x \in \inter(\csupp(Q))\cap \cup_{F\in\Omega_p} O_F.
    \end{align}

    Now, since the boundary of a convex set on $\mathcal{O}(\mathcal{E}, \Omega)$ also has $\nu$-measure zero, by dominated convergence, $\int e^\psi d\nu=1$.

    Next, we show that $\lambda\leq L(Q)$. 
    By Skorokhod’s representation theorem, there exists a probability space $(\Omega, \mathcal{A}, P)$ and the random variables on it $X_N$, $X$, such that $X_N\sim Q_N$, $X\sim Q$, and $\lim X_N = X~~(P-a.s.)$. Thus, if we let $H_N = -ad(X_N, x_{0}) + b - \psi_N(X_N)$,
\begin{align}
    \lambda = \lim \int \psi_N dQ_N &= \lim \left(\int (-ad(x, x_{0}) + b) dQ_N - E(H_N)\right) \nonumber \\
    &= -a \int d(x,x_{0}) dQ + b - \liminf E(H_N) \nonumber \\
    &\leq  -a \int d(x,x_{0}) dQ + b - E(\liminf H_N) & \nonumber \\
    &\leq -a \int d(x,x_{0}) dQ + b - E(-ad(X, x_{0}) + b- \psi(X))  \nonumber \\
    & \leq L(Q).
\end{align}

    Now we show the other side, $\lambda \geq L(Q)$. For $M>0$, we define
\begin{align}
    \psi_Q^{(M)}(x) = \sup_{y \in \mathcal{O}(\mathcal{E},\Omega)} \{\psi_Q(y) - Md(x,y)\}.
\end{align}
    It is obvious that $\psi_Q^{(M)}(x) \geq \psi_Q(x)$. By Lemma \ref{lem:lipschitz-majorization}, $\psi_Q^{(M)}$ satisfies the following three properties:
    \begin{enumerate}
        \item $\psi_Q^{(M)}$ is $M$-Lipschitz
        \item $\psi_Q^{(M)}\in \Phi$
        \item $\psi_Q^{(M)}(x)\searrow \psi_Q(x)$ pointwise.
    \end{enumerate}

     Now, by Skorokhod's representation theorem, there exists a probability space $(\Omega, \mathcal{A}, P)$ and random variables $Y_N, Y: \Omega \to X$ such that $Y_N \to Y~~(a.s.)$ and $Y_N\sim Q_N, Y\sim Q$. Thus, 

     \begin{align}
    \lambda &= \lim_{N\to\infty}\int \psi_{Q_N} dQ_N  \nonumber \\
    &\geq \limsup_{N\to\infty}\int \psi_Q^{(M)}dQ_N - \int \exp(\psi_Q^{(M)})d\nu + 1 \nonumber \\
    &= \limsup_{N\to\infty}\int \psi_Q^{(M)}(Y_N)dP - \int \exp(\psi_Q^{(M)})d\nu + 1 .
\end{align}

    Here, $\psi_Q^{(M)}(Y_N)$ is uniformly integrable. By the continuous mapping theorem, $\psi_Q^{(M)}(Y_N) \to \psi_Q^{(M)}(Y) ~~(a.s.)$. Also, by $M$-Lipschitzness of $\psi_Q^{(M)}$, the function $h(x) = \psi_Q^{(M)}(x)/(1+d(x,x_0))$ is bounded. Therefore, 
\begin{align}
    &E[|\psi_Q^{(M)}(Y_N)| \mathbf{1} \{|\psi_Q^{(M)}(Y_N)|\geq T\}] \\ =&  E\left[\frac{|\psi_Q^{(M)}(Y_N)|}{1+d(Y_N,x_0)} (1+d(Y_N,x_0)) \mathbf{1} \{|\psi_Q^{(M)}(Y_N)|\geq T\}\right] \nonumber \\
    \leq& R E\left[ (1+d(Y_N,x_0)) \mathbf{1} \{|\psi_Q^{(M)}(Y_N)|\geq T\} \right] \nonumber \\
    \leq&  E\left[ R(1+d(Y_N,x_0)) \mathbf{1} \{R(1 + d(Y_N,x_0))\geq T\} \right].
\end{align}

Now, $R(1+d(Y_N,x_0))$ is uniformly integrable because $d(Y_N,x_0)\to d(Y,x_0)$ almost surely and $\int d(Y_N,x_0) dP < \infty$, $\int d(Y_N,x_0) dP \to \int d(Y,x_0) dP$ imply $\int |d(Y_N,x_0)-d(Y,x_0)| dP \to 0$. Thus, $d(Y_N,x_0)\to d(Y,x_0)$ in $L^1$. $L^1$ convergence leads to the uniform integrability.

This shows the uniform integrability of $\psi_Q^{(M)}(Y_N)$ as well. By Vitali's theorem, $\psi_Q^{(M)}(Y_N)$ converges to $\psi_Q^{(M)}(Y)$ in $L^1$. Thus, 

\begin{align}
    \lambda &\geq \lim_{N\to\infty} \int \psi_Q^{(M)}(Y_N)dP - \int \exp(\psi_Q^{(M)})d\nu + 1 \nonumber \\
    &= \int \psi_Q^{(M)}(Y) dP - \int \exp(\psi_Q^{(M)})d\nu + 1 \nonumber \\
    &= \int \psi_Q^{(M)} dQ - \int \exp(\psi_Q^{(M)})d\nu + 1 .
\end{align}

Now, by monotone convergence, for $M\in\mathbb{N}$,
\begin{align}
     \int \psi_Q^{(M)} dQ &= -\int \left(\psi_Q^{(1)} - \psi_Q^{(M)}\right) dQ + \int \psi_Q^{(1)} dQ \nonumber \\
     &\xrightarrow{M\to \infty}  - \int \left(\psi_Q^{(1)} - \psi_Q\right) dQ + \int \psi_Q^{(1)} dQ  = \int \psi_Q dQ . 
\end{align}
Also, by dominated convergence, $\int \exp(\psi_Q^{(M)})d\nu \xrightarrow{M\to\infty} \int \exp(\psi_Q)d\nu = 1$.

Combination of all shows 
\begin{align}
    \lambda \geq \int \psi_Q dQ = L(Q).
\end{align}

This also shows that $\psi$ is the log-concave projection of $Q$ as well. By the uniqueness of the log-concave projection (Theorem \ref{th:uniqueness}), $\psi_Q = \psi$ $\nu$-almost surely and they agree on $\inter(\dom \psi) \cap \cup_{F\in\Omega_p} O_F$. In particular, by taking $\psi = \psi_Q$, we have
    \begin{align*}
        \limsup_{k\to\infty} \exp(\psi_{N(k)})(y) &\leq \exp(\psi_Q)(y)\leq \exp(-ad(y,x_{0}) + b) \quad \text{for all }y \in \mathcal{O}(\mathcal{E}, \Omega) \\
         \lim_{k\to\infty} \exp(\psi_{N(k)})(y) &= \exp(\psi_Q)(y) > 0 \quad \text{for all } y \in \inter(\csupp(Q))\cap \cup_{F\in\Omega_p} O_F.
    \end{align*}
    By dominated convergence,
    \begin{align}
        \lim_{k\to\infty} \int |\exp(\psi_{N(k)})(x) - \exp(\psi_Q)(x)|d\nu = 0.
    \end{align}

We only established the convergence result with respect to the subsequence $\{N(k)\}_k$, but all the results need to hold in the original sequence as well.

\end{proof}

For the proof of Proposition \ref{prop: weak_conv_log_conc}, we first need the following result about uniform convergence of closed convex sets. This is a result obtained by \citet{Rao1962-zt} in Euclidean space.

\begin{theorem}\label{th:uniform_convex}
    Suppose $\mu$ is a measure on $\mathcal{O}(\mathcal{E},\Omega)$ such that every convex set has $\mu$-null boundary. Then $\mu_n\xrightarrow[d]{}\mu$ if and only if 
    \begin{align}
        \sup_{C\in \mathcal{C}}|\mu_n(C) - \mu(C)| \to 0.
    \end{align}
    where $\mathcal{C}$ denotes the set of all closed convex sets.
\end{theorem}

\begin{proof}
    $\mu_n$ restricted to each nonnegative orthant weakly converges to $\mu$ restricted to the same orthant. The rest follows from the Euclidean result by \citet{Rao1962-zt}.
\end{proof}

\begin{proof}[Proof of Proposition \ref{prop: weak_conv_log_conc}]
    
    The proof mostly parallels that of the Euclidean case derived by \cite{Cule2010-uy}.
    \paragraph{(1)}
    
    Let $\mathcal{C}$ be the set of all convex sets. Then, 
    \begin{align}
        \sup_{C\in\mathcal{C}}\left| \int_{C} (f_N - f) d\nu \right| \to 0.
    \end{align}
    as $N\to\infty$ by Theorem \ref{th:uniform_convex}. 

    Also, the Lebesgue Differentiation Theorem can be generalized to this space:
    for almost every point $x_0$, for any nicely shrinking sets $\{B_i\}_i$ to $x_0$,
    \begin{align}
        \lim_{i\to\infty} \frac{1}{\nu(B_i)}\int_{B_i}|f(x)-f(x_0)|d\nu(x) = 0.
    \end{align}

    Let $E_1 = \{ x\mid \liminf f_N(x) < f(x) \}$ and suppose $\nu(E_1) > 0$. Then there exists a Lebesgue point $x_0$ of $f$ with $x_0\in E_1$, $f(x_0)<\infty$, and $x_0$ is in one of the positive orthants. Let $\epsilon = f(x_0) - \liminf f_N(x_0)$. We can find a sequence $f_{N_k}$ with $f_{N_k}(x_0) < f(x_0) - 3\epsilon/4$. For $\delta > 0$, define the convex sets
    \begin{align}
        D_{k,\delta} = \{x\in \bar{B}_{\delta}(x_0) \mid f_{N_k}(x)\geq f(x_0) - \epsilon/2\}.
    \end{align}
    By choosing $\delta$ sufficiently small such that the $\delta$-ball around $x_0$ is contained in the positive orthant, we have $\nu(B_\delta\backslash D_{k,\delta}) \geq \nu(B_{\delta})/2$ because of concavity. This means $B_\delta\backslash D_{k,\delta}$ is a nicely shrinking set with the same eccentricity bound for every $k$, thus, for sufficiently small $\delta$, 
    \begin{align}
        \frac{1}{\nu(B_{\delta}\backslash D_{k,\delta})} \left| \int_{B_\delta \backslash D_{k,\delta}} \{f(x_0) - f(x)\} d\nu(x) \right| \leq \frac{\epsilon}{8}.
    \end{align}

    Then
    \begin{align}
        &\int_{B_{\delta}} (f-f_{N_k}) d\nu \nonumber \\
        =& \int _{B_{\delta}\backslash D_{k,\delta}} \{f(x) - f(x_0) + f(x_0) - f_{N_k}(x))\} d\nu(x) + \int_{D_{k,\delta}} (f-f_{N_k}) d\nu \nonumber \\
        \geq& -\frac{\epsilon}{8}\nu(B_{\delta}) + \frac{\epsilon}{4}\nu(B_{\delta}) - \sup_{C\in\mathcal{C}}\left| \int_{C} (f-f_{N_k})d\nu \right| \nonumber \\
        \to& \frac{\epsilon}{8}\nu(B_{\delta}).
    \end{align}
    which is a contradiction. Hence $\nu(E_1) = 0$.

    Thus, without loss of generality, we may assume that $f\leq \liminf_{N}f_N$. But by Fatou's lemma, 
    \begin{align}
        1 = \int f \leq \int \liminf f_N \leq \liminf \int f_N = 1.
    \end{align}
    so we may assume that $f = \liminf f_N$. $\liminf f_N$ is log-concave.

    \paragraph{(2)}

    Suppose $\nu(E_2) > 0$, where $E_2 = \{x \mid \limsup f_N(x) > f(x)\}$. Because $f$ is continuous $\nu$-a.e., we can find $x_0\in E_2$ such that $x_0$ is in one of the positive orthants and $f$ is continuous at this point. Let $\epsilon_0 = \limsup f_N(x_0) - f(x_0)$ and $\epsilon = \min\{1, \epsilon_0\}$. We can find a subsequence $(f_{N_k})$ with $f_{N_k}(x_0) > f(x_0) + \frac{3\epsilon}{4}$ for all $k$. Define the convex sets
    \begin{align}
        \tilde{D}_{k,\delta} = \{x\in \bar{B}_{\delta}(x_0) \mid f_{N_k}(x) \geq f(x_0) + \epsilon/2\}.
    \end{align}

    \begin{itemize}
        \item[Case 1]

        Suppose that the sequence $\{\tilde{D}_{k,\delta} \mid k\in\mathbb{N}, \delta>0\}$ shrinks nicely to $x_0$ with the same eccentricity bound $\eta$. Then, for $\delta$ sufficiently small such that $|f(x)-f(x_0)| < \epsilon/4$ for all $x\in \bar{B}_{\delta}(x_0)$, for every $k$, 
        \begin{align}
            \int_{\tilde{D}_{k,\delta}} (f_{N_k} - f) \geq \frac{\nu(\tilde{D}_{k,\delta})\epsilon}{4} \geq \frac{\nu(\bar{B}_{\delta}(x_0))\eta\epsilon}{4},
        \end{align}
        which contradicts uniform convergence in the class of convex sets.

        \item[Case 2]
        Suppose no uniform eccentricity bound holds across $k$; i.e., 
        $\forall\,\eta\in(0,1],\,\forall\,\delta_0>0,\ \exists\,k,\ 0<\delta<\delta_0$
        with $\nu(\tilde D_{k,\delta})<\eta\,\nu(B_\delta(x_0))$.
        Assume $f(x_0)>0$ and, decreasing $\epsilon$ if needed, $f(x_0)>\epsilon/2$.
        If a change in $\epsilon$ makes them shrink nicely with the same eccentricity bound, then it reduces to Case 1. 
        Let $\delta_0>0$ such that $\bar{B}_{\delta_0}(x_0)$ is entirely contained in one positive orthant. Then for $\delta<\delta_0$ and for each $k$, the ratio $\nu(\tilde{D}_{k,\delta})/\nu(\bar{B}_{\delta}(x_0))$ has to be decreasing because of concavity. 
        
        If all of $\{\tilde{D}_{k,\delta} \mid \delta>0\}$ except for some finite $k$ have the same positive eccentricity bound, then simply those indices $k$ can be ignored, and the arguments given in Case 1 can be exploited. So we only have to consider the case where we can take a sequence $k(1)<k(2)<\ldots$ and $\delta$ such that

        \begin{align}
            \frac{\nu(\tilde{D}_{k(l),\delta})}{\nu(\bar{B}_{\delta}(x_0))} \leq \xi
        \end{align}
        for any given $\xi>0$.

        By concavity, for $\delta < \delta_0$, it is simple to see that $\nu(D_{k,\delta}) \leq K(\epsilon)\nu(\tilde{D}_{k,\delta})$ for some constant $K(\epsilon)$. We can also assume that $|f(x)-f(x_0)|\leq \epsilon/4$ for all $x\in \bar{B}_{\delta}(x_0)$. We can therefore conclude that for all $l$, 
        \begin{align}
            \int_{\bar{B}_{\delta}(x_0)} (f_{N_{k(l)}} - f) &= \int_{\bar{B}_{\delta}\backslash {D_{k(l),\delta}}} (f_{N_{k(l)}} - f) + \int_{ {D_{k(l),\delta}}} (f_{N_{k(l)}} - f) \nonumber \\
            &\leq -\frac{\epsilon}{4} \nu(\bar{B}_{\delta}\backslash D_{k(l),\delta}) + \sup_{C\in\mathcal{C}}\int_{C}(f_{N_{k(l)}} - f) \nonumber \\
            &\leq -\frac{\epsilon}{4}(1 - K\xi)\nu(\bar{B}_{\delta}(x_0)) + \sup_{C\in\mathcal{C}}\int_{C}(f_{N_{k(l)}} - f).
        \end{align}
        Thus by taking $\xi$ to be sufficiently small, we have  $\limsup_{l}\int_{\bar{B}_{\delta}(x_0)} (f_{N_{k(l)}} - f) < 0$, which is a contradiction.

        \item[Case 3] 
        Finally, suppose $f(x_0) = 0$. We only have to consider the case where $\nu(\{x \in E_2 \mid f(x)=0\}) > 0$, because otherwise, we should be able to find a point $x_0$ that satisfies the conditions of either case 1 or case 2. 
        There should be at least one positive orthant $O_F$ such that $\nu(\{x \in E_2 \mid f(x)=0\} \cap O_F) > 0$. Let $C_0 = \{x \in E_2 \mid f(x)=0\} \cap O_F$.  Since $\{f(x)>0\}$ is convex so its boundary has null $\nu$-measure, $D_0 = C_0 - \partial\{f(x)>0\}$ still has positive measure. Then, we can pick $p$ more points from $D_0$, $x_1, \ldots, x_{p}$ in such a way that $\nu(\mathrm{conv}(x_0, \ldots, x_p)) > 0$. Also, since $x_0\in D_0$, for sufficiently small $\delta$, and $x \in\bar{B}_{\delta}(x_0)$, $f(x)=0$. Then for sufficiently large $k$, at one of the $x_j$, $f_{N_k}(x_j)$ must approach $0$ because otherwise it contradicts the convergence $|\int_{\bar{B}_{\delta}(x_0)\cap \mathrm{conv}(x_0, \ldots, x_p)} f_{N_k} - f|\to 0$. This contradicts the assumption.
    \end{itemize}

    Hence we have $\nu(E_2) = 0$. This concludes (2).

    \paragraph{(3)}

    Let $\Theta = \{F\in\Omega_p\mid \int_{O_F}f~d\nu>0\}$.
    Write $\phi_N = \log f_N$ and $\phi = \log f$. For each $F_i\in\Theta$, we can find a point $x_i\in\mathrm{int}(\mathrm{dom}~\phi)\cap O_{F_i}$ such that the closed ball of radius $\eta>0$ about $x_i$ is contained in $\mathrm{int}(\mathrm{dom}~\phi)\cap O_{F_i}$. (This is because $\mathrm{dom}\phi$ is always convex, and for a convex set to have a positive measure, its interior must be nonempty. $\eta$ can be taken globally.)

    Fix $a\in(0,a_0)$ and let $\delta = a_0 - a$. Then we can find $R>0$ such that $(d(x,x_i))^{-1}\{\phi(x) - \phi(x_i)\} \leq -(a + \frac{3\delta}{4})$ for all $x\in O_{F_i}$ with $d(x,x_i)\geq R/2$. By utilizing the same argument as in the Euclidean case, we can claim that there exists $N^{(i)}$ such that 
    \begin{align}
        \frac{\phi_N(x) - \phi_N(x_i)}{d(x,x_i)} \leq -\left(a + \frac{\delta}{4}\right).
    \end{align}
    for all $x\in O_{F_{i}}$ with $d(x,x_i)\geq R$ and $N\geq N^{(i)}$. This means that if we set $N_0 = \max_{i}N^{(i)}$, there exists $b\in\mathbb{R}$ such that  $\sup_{N\geq N_0, x\in \cup_{F\in\Theta} O_F} f_N(x) \leq \exp(-(a + \delta/4)d(x,0) + b)$. 

    Now, from (2) and the dominated convergence theorem, we can conclude that
    \begin{align}
        \int \exp(ad(x,0))|f_N(x) - f(x)|d\nu(x) \to 0.
    \end{align}

Finally, suppose that $f$ is continuous and $\mathcal{O}(\mathcal{E},\Omega)$ has the geodesic extension property. Then, on any convex closed set $S$ contained in $\inter(\dom \phi)$, $f_N$ converges to $f$ uniformly. As in the Euclidean case \cite{Cule2010-uy}, we can derive the uniform convergence result as follows. First, for any $\epsilon\in(0,1)$, from the previous argument, we can take a large enough $R>0$ such that for all $x$ with $d(x,0)\geq R$ and for all $N\geq N_0$, $f(x)+f_N(x) \leq \epsilon e^{-ad(x,0)}$. Let $S=\{x\in \bar{B}_R(0) \mid f(x)\geq \epsilon e^{-aR}/2 \}$. This $S$ is closed and convex.  By uniform convergence of $f_N$ to $f$ in $S$, for large enough $N$, $\sup_{x\in S} |f_N(x) - f(x)| \leq \epsilon e^{-aR}$. Combining all, we get
\begin{align}
    \sup_{x\in \mathcal{O}(\mathcal{E},\Omega)} e^{ad(x,0)}|f_N(x)-f(x)| \to 0.
\end{align}

\end{proof}

\section{Proof of Results in Section 4.}
\label{sec:proof-kernel-consistency}
In this subsection, we prove Theorem \ref{th:kde_consistency}.

\subsection{Proof strategy}
To prove the consistency of the kernel density estimator with the modified kernel $K_2$, we consider the following smoothed version of the density function $f$:
\begin{equation}
    f_h(x) = c(x,h) \int_{\mathcal{O}(\mathcal{E},\Omega)} \frac{1}{(2\pi)^{p/2}h^p} \exp\left(-\frac{d(y,x)^2}{2h^2}\right) f(y) d\nu(y),
\end{equation}
where $c(x,h)  = (2\pi)^{p/2} h^p C(x,h)$.

Our proof proceeds in two main steps:

\begin{enumerate}
    \item Show that $f_h(x)$ converges uniformly to $f(x)$ as $h \to 0$, that is,

    \begin{equation}
        \sup_{x \in \mathcal{O}(\mathcal{E},\Omega)} |f_h(x) - f(x)| \to 0 \quad \text{as } h \to 0.
    \end{equation}

    \item Show that the difference between the KDE and the smoothed density function converges to zero uniformly in $x$, that is,

    \begin{equation}
        \sup_{x \in \mathcal{O}(\mathcal{E},\Omega)} |\hat{f}_{N,h_N}(x) - f_{h_N}(x)| \to 0 \quad \text{almost surely as } N \to \infty.
    \end{equation}
\end{enumerate}

\subsection{Proof of uniform convergence of $f_h(x)$ to $f(x)$}

We first show that $f_h(x)$ converges uniformly to $f(x)$ as $h \to 0$. Since $f$ is uniformly continuous and bounded, for any $\varepsilon>0$ there exists $\delta>0$ such that $d(x,y)<\delta \Rightarrow |f(x)-f(y)|<\varepsilon/2$. 
Set
\[
w_{x,h}(y):=\frac{c(x,h)}{(2\pi)^{p/2}h^p}\exp\!\Big(-\frac{d(y,x)^2}{2h^2}\Big),
\qquad \int_{\mathcal{O}(\mathcal{E},\Omega)} w_{x,h}(y)\,d\nu(y)=1.
\]
Then
\begin{align}
|f_h(x)-f(x)|
=&\left|\int_{\mathcal{O}(\mathcal{E},\Omega)} w_{x,h}(y)\,(f(y)-f(x))\,d\nu(y)\right| \\
\le& \int_{\mathcal{O}(\mathcal{E},\Omega)} w_{x,h}(y)\,|f(y)-f(x)|\,d\nu(y) \\
=& \int_{B_\delta(x)} w_{x,h}(y)\,|f(y)-f(x)|\,d\nu(y) \nonumber \\
  &+ \int_{\mathcal{O}(\mathcal{E},\Omega)\setminus B_\delta(x)} w_{x,h}(y)\,|f(y)-f(x)|\,d\nu(y).
\end{align}
For the first term, $|f(y)-f(x)|<\varepsilon/2$ and $\int_{B_\delta(x)} w_{x,h}\,d\nu \le 1$, so it is $\le \varepsilon/2$.
For the second term, let 
\[
\eta(h,\delta):=\sup_{x}\int_{d(y,x)\ge\delta} w_{x,h}(y)\,d\nu(y)\xrightarrow[h\downarrow 0]{}0.
\]
Then $|f(y)-f(x)|\le 2\|f\|_\infty$ gives a bound $2\|f\|_\infty\,\eta(h,\delta)$.
Choosing $h$ so that $\eta(h,\delta)<\varepsilon/(4\|f\|_\infty)$ yields
\[
\sup_{x\in\mathcal{O}(\mathcal{E},\Omega)} |f_h(x)-f(x)| \le \frac{\varepsilon}{2}+2\|f\|_\infty\,\eta(h,\delta) < \varepsilon.
\]

\subsection{Convergence of $\hat{f}_{N,h}(x)$ to $f_h(x)$}

We first provide some devices used in the proof.

\paragraph{Montgomery-Smith's Inequality}
\begin{lemma}[Montgomery-Smith Inequality \cite{Montgomery-Smith1993-zt}]\label{lem:montgomery-smith}
Let $\{X_i\}_{i=1}^N$ be i.i.d. random variables taking values in a Banach space. Let $S_k = \sum_{i=1}^k X_i$. Then, for any $t > 0$ and $1 \leq k \leq N$,

\begin{equation}
    \Pr\left( \|S_k\| > t \right) \leq 3 \Pr\left( \|S_N\| > \frac{t}{10} \right).
\end{equation}
\end{lemma}

Furthermore, the following inequality holds:

\begin{lemma}[\cite{Protter1995-sm}]
Under the same conditions as Lemma \ref{lem:montgomery-smith}
    \begin{align}
        \Pr\left(\max_{1\leq k\leq N}\|S_k\| > t\right) \leq 3 \max_{1\leq k\leq N}\Pr\left(\|S_k\| > \frac{t}{3}\right).
    \end{align}
\end{lemma}

These two results combined, the following inequality holds:
\begin{align}
        \Pr\left(\max_{1\leq k\leq N}\|S_k\| > t\right) \leq 9\Pr\left(\|S_N\| > \frac{t}{30}\right).
    \end{align}

\paragraph{Talagrand's Inequality}
The following is a version of Talagrand's inequality \cite{Talagrand1994-md, Talagrand1996-tj} taken from \cite{Gine2001-be}.
\begin{theorem}[\cite{Gine2001-be}] \label{th:talagrand}
    Let $\{X_i\}_{i=1}^N$ denote an i.i.d. sample from a probability measure $P$. Let $\mathcal{F}$ be a uniformly bounded VC-subgraph function class, and let $\sigma^2$ and $U$ satisfy $\sigma^2 \geq \sup_{f\in\mathcal{F}} \mathrm{Var}_P(f)$, $U \geq \sup_{f\in\mathcal{F}}\|f\|_{\infty}$, $0<\sigma<U/2$, and $\sqrt{N}\sigma\geq U\sqrt{\log(U/\sigma)}$, where $\mathrm{E}_P$ denotes the expectation with respect to $P$. Then, there exist constants $C$ and $L$ (which depend on the VC properties of $\mathcal{F}$) such that for $C_2 \geq C$,
\begin{align}
    &\Pr\left\{ \sup_{f\in\mathcal{F}}\left|\sum_{i=1}^N (f(X_i) - \mathrm{E}_P f)\right| > C_2\sqrt{N}\sigma\sqrt{\log\left(\frac{U}{\sigma}\right)} \right\} \nonumber \\
   & \leq L\exp\left(-\frac{C_2\log(1+C_2/(4L))}{L}\log\left(\frac{U}{\sigma}\right)\right).
\end{align}
\end{theorem}

Now we show that $\hat{f}_{N,h}(x)$ converges uniformly to $f_h(x)$. Consider the difference:

\begin{equation}
    |\hat{f}_{N,h}(x) - f_h(x)| = \left| \frac{1}{N} \sum_{i=1}^N K_2(x \mid X_i,h) - \int_{\mathcal{O}(\mathcal{E},\Omega)} K_2(x \mid y,h) f(y) d\nu(y) \right|.
\end{equation}

This can be rewritten as:

\begin{equation}
    |\hat{f}_{N,h}(x) - f_h(x)| = c(x,h) \left| (\mathbb{E}_N - \mathrm{E}_P) \frac{1}{h^p(2\pi)^{p/2}} \exp\left( -\frac{d(x,X)^2}{2h^2} \right) \right|,
\end{equation}
where $\mathbb{E}_N$ denotes the expectation w.r.t. the empirical measure.

We define the following class of functions:

\begin{equation}
    \mathcal{F} = \left\{ k(y \mid x,h) \coloneqq  \frac{1}{(2\pi)^{p/2}}\exp\left( -\frac{d(y,x)^2}{2h^2} \right) \middle| \ x \in \mathcal{O}(\mathcal{E},\Omega), h > 0 \right\}.
\end{equation}

We show in the subsequent subsection that the function class $\mathcal{F}$ is VC-subgraph with an envelope function $1$. Note also that the supremum over functions in $\mathcal{F}$ is measurable. This is because of the continuity of the kernel with respect to the center $x$ and $h>0$, and that the tree space has a dense countable subset.

Set $N_k=2^k$. For sufficiently large $k$,
\begin{align}
h_{N_k}^p \log h_{N_k}^{-1} \;\le\; h_{N_{k-1}}^p \log h_{N_{k-1}}^{-1}.
\end{align}
Also, $c(x,h)$ is uniformly bounded in $(x,h)$; let this bound be $M$.
Define the subclass
\begin{align}
\mathcal{F}_k
\;=\;
\big\{\, k(y \mid x,h) \;:\; x \in \mathcal{O}(\mathcal{E},\Omega),\; h_{N_k} \le h < h_{N_{k-1}} \,\big\}.
\end{align}
This is again a uniformly bounded and VC-subgraph class of functions. Then
\begin{align*}
&\Pr\!\Bigg( \max_{N_{k-1} < N \le N_k}
\sqrt{\frac{N h_N^p}{\log h_N^{-1}}}\,
\|\hat{f}_{N,h_N} - f_{h_N}\|_{\infty} > \lambda \Bigg) \\
\le& \Pr\!\Bigg( \max_{N_{k-1} < N \le N_k}
\frac{M}{\sqrt{N h_N^p \log h_N^{-1}}}\; \times \\
&\qquad\qquad\qquad\qquad
\sup_{x\in\mathcal{O}(\mathcal{E},\Omega)}
\Bigg| \sum_{i=1}^N \Big(k(X_i \mid x, h_N)
- \mathrm{E}_P k(X \mid x, h_N)\Big) \Bigg| > \lambda \Bigg) \\
\le& \Pr\!\Bigg(
\frac{M}{\sqrt{N_{k-1}\, h_{N_k}^p \log h_{N_{k-1}}^{-1}}}\,
\max_{1 \le N \le N_k}
\sup_{f\in\mathcal F_k}
\Bigg| \sum_{i=1}^N \Big(f(X_i) - \mathrm{E}_P f\Big) \Bigg|
> \lambda \Bigg) \\
\le& 9\, \Pr\!\Bigg(
\frac{M}{\sqrt{N_{k-1}\, h_{N_k}^p \log h_{N_{k-1}}^{-1}}}\,
\sup_{f\in\mathcal F_k}
\Bigg| \sum_{i=1}^{N_k} \Big(f(X_i) - \mathrm{E}_P f\Big) \Bigg|
> \frac{\lambda}{30} \Bigg).
\end{align*}

Set $C_\ast:=|\Omega_p|\,2^{-p}\pi^{-p/2}$, $U_k:=1$, and $\sigma_k^2:=C_\ast\,h_{N_{k-1}}^{p}\,\|f\|_\infty$. With this choice, all conditions of Theorem \ref{th:talagrand} are satisfied for all sufficiently large $k$: for any $f\in\mathcal F_k$ one has $\mathrm{Var}_P(f)\le \mathbb E_P f^2 \le \|f\|_\infty\int k^2\,d\nu \le C_\ast \|f\|_\infty h^p \le \sigma_k^2$; $\mathcal F_k$ is uniformly bounded with envelope $U_k=1$ and is VC-subgraph; since $h_{N_{k-1}}\downarrow0$, we have $0<\sigma_k<U_k/2$ for large $k$; and the bandwidth regime gives $N_k\sigma_k^2/\log(1/\sigma_k)\to\infty$, hence $\sqrt{N_k}\sigma_k \ge U_k\sqrt{\log(U_k/\sigma_k)}$ eventually.

Furthermore, there exists a constant $A>0$ (depending only on $p$, $\check c$, $C_\ast$, and $\|f\|_\infty$) such that, for all sufficiently large $k$,
\begin{align}
\sqrt{N_k}\,\sigma_k\,\sqrt{\log\!\left(\frac{U_k}{\sigma_k}\right)}
\;\le\;
A\,\sqrt{\,N_{k-1}\,h_{N_k}^{p}\,\log\!\big(h_{N_k}^{-1}\big)}. \label{eq:bound_nsl}
\end{align}

Choose $\lambda := 30\,M\,C_2\,A$ (with $C_2$ as in Theorem~\ref{th:talagrand}). Then, by combining the above results,
\begin{align}
&\Pr\!\left( \max_{N_{k-1} < N \le N_k}
\sqrt{\frac{N h_N^p}{\log h_N^{-1}}}\,
\|\hat{f}_{N,h_N} - f_{h_N}\|_{\infty}
> 30MC_2A \right)
\label{eq:prob} \\
&\qquad\le\; 9\,L\,\exp\!\left(
-\frac{C_2 \log\!\big( 1 + \tfrac{C_2}{4L} \big)}{L}\,
\log\!\left( \frac{1}{\sigma_k} \right) \right).
\end{align}
From our assumptions,
\[
\frac{\log (1/\sigma_k)}{\log k}
=\frac{\tfrac12\log(\|f\|_\infty^{-1}) - \tfrac12\log C_\ast
+ \tfrac{p}{2}\log\!\big(h_{2^{k-1}}^{-1}\big)}
{\log\log 2^k}\;\longrightarrow\;\infty.
\]
Hence there exists $k_0$ such that for all $k\ge k_0$,
\[
\frac{C_2 \log\!\big( 1 + \tfrac{C_2}{4L} \big)}{L}\,
\log\!\Big( \tfrac{1}{\sigma_k} \Big) \;\ge\; 2\log k,
\]
and therefore
\[
\sum_{k=k_0}^{\infty}
\exp\!\left(
-\frac{C_2 \log\!\left( 1 + \tfrac{C_2}{4L} \right)}{L}\,
\log\!\left( \tfrac{1}{\sigma_k} \right) \right)
\;\le\; \sum_{k=k_0}^{\infty} \frac{1}{k^{2}} \;<\;\infty.
\]
Thus, the sequence of probabilities in \eqref{eq:prob} is summable.

Let $Y \coloneqq \limsup Y_N$, where $Y_N \coloneqq \sqrt{\dfrac{N h_N^p}{\log h_N^{-1}}} \, \|\hat{f}_{N,h_N} - f_{h_N}\|_{\infty}$. By the Borel--Cantelli lemma, we have $P(Y > \lambda) = 0$.
Therefore, in particular,
\begin{align}
    \|\hat{f}_{N,h_N} - f_{h_N}\|_{\infty} \to 0 \quad \text{almost surely}.
\end{align}

\subsection{Proof that $\mathcal{F}$ is VC-subgraph}\label{subsec:VC-subgraph}

The class $\mathcal{F}$ is VC-subgraph if the following class is VC-subgraph:
\begin{align}
    \mathcal{D} = \left\{ d(y,x)/h  \middle| \ x \in \mathcal{O}(\mathcal{E},\Omega), h>0\right\}.
\end{align}

This is because of the monotonicity of the function $x\mapsto\exp(-x^2)$ for $x\geq0$. We will use some geometrical properties of the geodesics to prove that this $\mathcal{D}$ is VC-subgraph.

Let $P_k$ denote the set of ordered $(k+1)$-partitions of $\mathcal{E}$. For $A\in P_k$, we denote by $A_i~(i=0,\ldots,k)$ the $i$-th partition. Let $\mathcal{E}(t)$ denote the set of active axes of a point $t$ (those with strictly positive coordinates). For axis sets $(\mathcal{E}(t),\mathcal{E}(t^\prime))$ and $(A,B)\in P_k \times P_k$, we define $(A,B)\cap (\mathcal{E}(t),\mathcal{E}(t^\prime))$ to be the pair of sequences $((A_0\cap\mathcal{E}(t), \ldots, A_k\cap\mathcal{E}(t)), (B_0\cap\mathcal{E}(t^\prime), \ldots, B_k\cap\mathcal{E}(t^\prime)))$. We define $(A,B)\in P_k\times P_k$ to be {\it proper} for points $t$ and $t^\prime$, if the following two conditions are satisfied:
\begin{align}
    \text{(P1)} \quad & \text{$A_0\cap\mathcal{E}(t) = B_0\cap\mathcal{E}(t^\prime)$, and for all $i>j>0$}, \nonumber \\ &\text{$A_i\cap\mathcal{E}(t)$ and $B_j\cap\mathcal{E}(t^\prime)$ are nonempty and compatible},\nonumber \\
    \text{(P2)} \quad &  \frac{\sum_{l\in A_1} {t_l^2}}{\sum_{m\in B_1} {t_m^\prime}^2} \leq \frac{\sum_{l\in A_2} {t_l^2}}{\sum_{m\in B_2} {t_m^\prime}^2} \leq \cdots \leq \frac{\sum_{l\in A_k} {t_l^2}}{\sum_{m\in B_k} {t_m^\prime}^2},\nonumber
\end{align}
where  $t_j$ (resp. $t_m^\prime$) denotes the $j$-th (resp. $m$-th) coordinate of $t$ (resp. $t^\prime$). 

Now, define $L_{A, B}(t,t^\prime)$ to be the length of the shortest path from $t$ to $t^\prime$ going through the support sequence if $(A, B)$ is proper for $t$ and $t^\prime$, and otherwise the distance of the cone-path. We then have the following relationship:
\begin{align}
    d(t,t^\prime) = \min_{k=0,\ldots,p} \min_{(A,B)\in P_k\times P_k} L_{A,B}(t,t^\prime).
\end{align}

We can further decompose $L_{A,B}$ into the following three parts:
\begin{align}
    L_{A,B}(t,t^\prime) =& \max\{f_{A,B}(t\mid t^\prime), g_{A,B}(t\mid t^\prime), h_{A,B}(t\mid t^\prime)\},
\end{align}
where,
\begin{align}
    f_{A,B}(t\mid t^\prime) =& \sqrt{\sum_{j\in A_0}(t_j - t_j^\prime)^2 + \sum_{i=1}^k \left(\sqrt{\sum_{l\in A_i}t_l^2} + \sqrt{\sum_{m\in B_i} {t_m^\prime}^2}\right)^2}, \\
    g_{A,B}(t\mid t^\prime) =& \left(1 - \prod_{i=1}^{k-1} \mathbf{1}\left(  \frac{\sum_{l\in A_i} {t_l^2}}{\sum_{m\in B_i} {t_m^\prime}^2} \leq \frac{\sum_{l\in A_{i+1}} {t_l^2}}{\sum_{m\in B_{i+1}} {t_m^\prime}^2} \right)\right) \times d_{\mathrm{cone}}(t,t^\prime), \\
    h_{A,B}(t\mid t^\prime) =& \left(1 - \mathbf{1}(\text{($(A,B)\cap(\mathcal{E}(t), \mathcal{E}(t^\prime))$) satisfies (P1)})\right) \times d_{\mathrm{cone}}(t,t^\prime).
\end{align}
Here, $d_{\mathrm{cone}}(t,t^\prime)$ represents the cone-path distance between two points $t$ and $t^\prime$:
\begin{align}
    d_{\mathrm{cone}}(t,t^\prime) = \sqrt{\sum_{l\in\mathcal{E}(t)} t_l^2} + \sqrt{\sum_{m\in\mathcal{E}(t^\prime)} {t_m^\prime}^2}.
\end{align}
Note that $f_{A,B}(t\mid t^\prime)\leq d_{\mathrm{cone}}(t,t^\prime)$. We adopt the convention that an empty product equals $1$ (so $g_{A,B}= 0$ when $k=1$).

The function $f_{A, B}$ gives the length of the shortest path from $t$ to $t^\prime$ when the path $(A, B)$ is proper. $h_{A,B}$ returns 0 if $(A,B)\cap(\mathcal{E}(t),\mathcal{E}(t^\prime))$ is a valid support sequence, and otherwise $d_{\mathrm{cone}}(t,t^\prime)$. Similarly, $g_{A,B}$ returns 0 if $(A,B)$ satisfies (P2), and otherwise $d_{\mathrm{cone}}(t,t^\prime)$.

Combining all of the above results, we can see that $\mathcal{D}$ is VC-subgraph if all of the following classes are VC-subgraph:
\begin{align}
    F_{A,B} =& \{t\mapsto f_{A,B}(t,t^\prime)/h \mid t^\prime \in \mathcal{O}(\mathcal{E},\Omega), h>0\},  \\
    G_{A,B} =& \{t\mapsto g_{A,B}(t,t^\prime)/h \mid t^\prime \in \mathcal{O}(\mathcal{E},\Omega), h>0\},  \\
    H_{A,B} =& \{t\mapsto h_{A,B}(t,t^\prime)/h \mid t^\prime \in \mathcal{O}(\mathcal{E},\Omega), h>0\}. 
\end{align}

We prove these in the following.
First, since the orthant space is a finite union of top-dimensional nonnegative orthants, the VC-subgraph dimension of a function class on the whole space is at most the sum of the VC-subgraph dimensions of its restrictions to the individual orthants. 
This observation leads to the conclusion that $t\mapsto f_{A,B}(t\mid t^\prime)/h$ and $t\mapsto d_{\mathrm{cone}}(t,t^\prime)/h$ are VC-subgraph. For $h_{A,B}(t,t^\prime)/h$, on each orthant it coincides pointwise either with the zero function (when (P1) holds) or with $d_{\mathrm{cone}}(t,t^\prime)/h$ (when (P1) fails), so it is also VC-subgraph.

For the function class $G_{A,B}$, let $d_C$ denote the VC-subgraph dimension of 
\[
\left\{ t \mapsto \frac{d_{\mathrm{cone}}(t,t^\prime)}{h} \ \middle| \ t^\prime \in \mathcal{O}(\mathcal{E},\Omega),\ h>0 \right\}.
\]
We show that the VC-subgraph dimension of $G_{A,B}$ is finite. First, let
\begin{align}
\theta_0(t,t^\prime) &:= 0,\\
\theta_i(t,t^\prime) &:= 1-\mathbf{1}\left(
\frac{\sum_{l \in A_i} t_l^2}{\sum_{l \in A_{i+1}} t_l^2}
\le
\frac{\sum_{m \in B_i} {t_m^\prime}^{2}}{\sum_{m \in B_{i+1}} {t_m^\prime}^{2}}
\right), \quad i=1,\ldots,k-1,
\end{align}
Then,
\begin{align}
g_{A,B}(t \mid t^\prime)
&= \left( \max_{i=0,\ldots,k-1} \theta_i(t,t^\prime) \right) d_{\mathrm{cone}}(t,t^\prime).
\end{align}
It suffices to bound the VC-subgraph dimension for a single index $i (\geq 1)$:
\begin{align}
\mathcal{G}_i
&:= \Big\{\, t \mapsto \theta_i(t,t^\prime)\,\frac{d_{\mathrm{cone}}(t,t^\prime)}{h}
\ \Big|\ t^\prime \in \mathcal{O}(\mathcal{E},\Omega),\ h>0 \,\Big\} \label{eq:Gi-theta}\\
&= \Big\{\, t \mapsto \Big( 1 - \mathbf{1}\!\big( c_i(t) \le b_i(t^\prime) \big) \Big)
\frac{d_{\mathrm{cone}}(t,t^\prime)}{h}
\ \Big|\ t^\prime \in \mathcal{O}(\mathcal{E},\Omega),\ h>0 \,\Big\}, \label{eq:Gi-cb}
\end{align}
where
\begin{align}
c_i(t) := \frac{\sum_{l \in A_i} t_l^2}{\sum_{l \in A_{i+1}} t_l^2},
\qquad
b_i(t^\prime) := \frac{\sum_{m \in B_i} {t_m^\prime}^{2}}{\sum_{m \in B_{i+1}} {t_m^\prime}^{2}}.
\end{align}

Consider any set of $d_C+2$ labeled points $\left(t^{(1)},\mu^{(1)}\right),\ldots,\left(t^{(d_C+2)},\mu^{(d_C+2)}\right)$, and reorder the $t^{(j)}$ so that
\begin{align}
c_i\left(t^{(1)}\right) \ge \cdots \ge c_i\left(t^{(d_C+2)}\right).
\end{align}
If $\mu^{(d_C+2)} \leq 0$, then this last point is always included in the subgraph $\{z<\cdot\}$ for every member of $\mathcal{G}_i$, thus the set is not shattered.
If $\mu^{(d_C+2)}>0$, then to include the last point we must have
\begin{align}
1-\mathbf{1}\left(c_i\!\left(t^{(d_C+2)}\right)\le b_i\!\left(t^\prime\right)\right)=1
\quad\text{i.e.,}\quad
c_i\!\left(t^{(d_C+2)}\right)>b_i\!\left(t^\prime\right).
\end{align}
Since $c_i\!\left(t^{(1)}\right)\ge\cdots\ge c_i\!\left(t^{(d_C+2)}\right)$, it follows that
\begin{align}
c_i\!\left(t^{(j)}\right)\ge c_i\!\left(t^{(d_C+2)}\right)>b_i\!\left(t^\prime\right)
\quad\text{for } j=1,\ldots,d_C+1,
\end{align}
and thus
\begin{align}
1-\mathbf{1}\left(c_i\!\left(t^{(j)}\right)\le b_i\!\left(t^\prime\right)\right)=1
\quad\text{for } j=1,\ldots,d_C+1.
\end{align}
Therefore, on these $d_C+1$ points the function reduces to
\begin{align}
t\longmapsto \frac{d_{\mathrm{cone}}(t,t^\prime)}{h}.
\end{align}
By the definition of $d_C$, the class $\left\{d_{\mathrm{cone}}(\cdot,t^\prime)/h\right\}$ cannot shatter $d_C+1$ points. Hence, once the last point is included, the remaining $d_C+1$ points cannot be labeled arbitrarily. Therefore, no set of size $d_C+2$ is shattered by $\mathcal{G}_i$, and hence the VC-subgraph dimension of $\mathcal{G}_i$ is finite. Since $G_{A,B}=\max_{i=1,\ldots,k-1} g^{(i)}_{A,B}$ is a finite maximum of classes with finite VC-subgraph dimension, $G_{A,B}$ has finite VC-subgraph dimension as well.


All of the above results imply that $\mathcal{D}$ is VC-subgraph, hence $\mathcal{F}$ is VC-subgraph as well. Since $\mathcal{F}$ has envelope $U\equiv 1$, the standard entropy bound for uniformly bounded VC-subgraph classes yields
\begin{align}
\sup_{Q} N\!\left(\epsilon,\ \mathcal{F},\ L_2(Q)\right)
\;\le\; \left(\frac{A}{\epsilon}\right)^{v},
\end{align}
for some constants $A,v>0$ depending only on the VC characteristics of $\mathcal{F}$.

\end{appendix}

\section*{Funding}
Yuki Takazawa was supported by JSPS KAKENHI (Grant Numbers 22KJ1131, 25K24364). Tomonari Sei was supported by JSPS KAKENHI (Grant Number 21K11781).

\bibliographystyle{unsrtnat}
\bibliography{LCDTreeSpacePaper}  
\end{document}